\documentclass[11pt]{amsart}

\usepackage[utf8]{inputenc}
\usepackage[T1]{fontenc}
\usepackage[text={426pt,605pt}, centering]{geometry}
\usepackage{stmaryrd}
\usepackage{braket}
\usepackage{mathtools}
\usepackage{mathrsfs}
\usepackage{dsfont}
\usepackage[shortlabels]{enumitem}
\usepackage{tikz-cd}

\theoremstyle{plain}
\newtheorem{theorem}{Theorem}[section]
\newtheorem{corollary}[theorem]{Corollary}
\newtheorem{proposition}[theorem]{Proposition}
\newtheorem{lemma}[theorem]{Lemma}
\newtheorem{definition}[theorem]{Definition}

\newtheorem{assum}{Assumption}
\theoremstyle{remark}
\newtheorem{remark}[theorem]{Remark}
\newtheorem{example}[theorem]{Example}

\newlist{propenum}{enumerate}{1} 
\setlist[propenum]{label=(\roman*)}

\DeclarePairedDelimiter\abs{\lvert}{\rvert}   
\DeclarePairedDelimiter\norm{\lVert}{\rVert}  

\SetSymbolFont{stmry}{bold}{U}{stmry}{m}{n}

\title{An Infinite-Dimensional SIS Model}

\date{\today}

\author{Jean-François Delmas}
\address{Jean-François Delmas,
 CERMICS, École des Ponts, France}
\email{jean-francois.delmas@cermics.enpc.fr}

\author{Dylan Dronnier}
\address{Dylan Dronnier,
 CERMICS, École des Ponts, France}
\email{dylan.dronnier@enpc.fr}

\author{Pierre-André Zitt}
\address{Pierre-André Zitt, LAMA, Université Gustave Eiffel, France}
\email{pierre-andre.zitt@u-pem.fr}



\thanks{The authors are very grateful to Jean-Stéphane Dhersin for the discussion on the general SIS
  model and in particular on models for the vaccination and lockdown policies.}

\begin{document}


\begin{abstract}
  In this article, we introduce an infinite-dimensional deterministic SIS model which takes into
  account the heterogeneity of the infections and the social network among a large population.  We
  study the long-time behavior of the dynamic.  We identify the basic reproduction number $R_0$
  which determines whether there exists a stable endemic steady state (super-critical case: $R_0>1$)
  or if the only equilibrium is disease-free (critical and sub-critical case: $R_0\leq1$).  As an
  application of this general study, we prove that the so-called ``leaky'' and ``all-or-nothing''
  vaccination mechanism have the same effect on $R_0$.  This framework is also very natural and
  intuitive to model lockdown policies and study their impact.
\end{abstract}


\maketitle


\section{Introduction}
\subsection{Motivation}
\subsubsection{The SIS model} Some infections do not confer any long-lasting immunity. With such
infections, individuals become susceptible again once they have recovered from the disease. The
simplest deterministic way to model this kind of epidemics in a constant size population is the
following system of ordinary differential equations, introduced by Kermack and McKendrick in
\cite{kermack_contribution_1927} and known as the SIS (susceptible/infected/susceptible) model:
\[
  \left\{
    \begin{array}{l}
      \dot{S} = - \frac{K}{N} I S + \gamma I, \\
      \\
      \dot{I} = \frac{K}{N} I S - \gamma I,
    \end{array} 
  \right.
\]
where $S=S(t)$ and $I=I(t)$ are the number of susceptible and infected individuals, the total size
$N=S(t)+I(t)$ of the population is constant in time, and $K$ and $\gamma$ are two positive numbers
which represent the infectiousness and the recovery rate of the disease. The proportion
$U(t) = I(t)/N$ of infected individuals in the population evolves autonomously, according to:
\begin{equation}\label{eq:one-group}
  \dot{U} = (1 - U) K U - \gamma U.
\end{equation}
Looking at a time change of $U$ given by $V(t)=U(t/\gamma)$ and setting $R_0=K/\gamma$, one gets
that $ \dot{V} = (1 - V) R_0 V - V$. The parameter $R_0$ can be interpreted as the number of
infected individuals one infected individual generates on average over the course of its infectious
period, in an otherwise uninfected population.  This basic reproduction number was first introduced
by Macdonald \cite{macdonald_analysis_1952}, and appears in a large class of models in epidemiology,
see the monograph \cite{brauer-castillo} from Brauer and Castillo-Chavez .  The ordinary
differential equation in $V$ is well-posed and admits an explicit solution.  If $V(0) = 0$, then
$V(t)=0$ for all $t$: as $V$ represents the proportion of infected individuals, this constant
solution is called the \emph{disease-free equilibrium}.  Now assume $V(0)=V_0\in (0, 1]$. If
$R_0\neq 1$, the proportion of infected individuals in the population for $t\geq 0$ is given by:
\[
  U(\gamma t) =V(t) = \frac{R_0 -1}{R_0 + ((1-R_0)/V_0 -
    R_0)\mathrm{e}^{(1- R_0)t}} \cdot
\]
If $R_0=1$, then the proportion of infected individuals in the population is given by:
\[
  U(\gamma t) = V(t)=\frac{1}{(1/V_0) + t}\cdot
\]
Hence, one can identify three possible longtime behaviors for the dynamical system:
\begin{description}
\item[Sub-critical regime] If $R_0<1$, $U(t)$ converges exponentially fast to $0$, and the only
  equilibrium is the disease-free solution $U(t)=0$.
\item[Critical regime] If $R_0=1$, $U(t)$ still converges to $0$ but not exponentially. The
  disease-free equilibrium is still the only one.
\item[Super-critical regime] If $R_0>1$, the constant solution $0$ becomes unstable and another
  equilibrium appears, $G^* = 1-R_0^{-1}$.  This equilibrium is called \emph{endemic}, and is
  globally stable in the sense that $U(t)$ converges towards $G^*$ for all initial positive
  conditions.
\end{description}

\subsubsection{The multidimensional Lajmanovich Yorke extension}
In a pioneering paper \cite{lajmanovich1976deterministic}, Lajmanovich and Yorke introduced an
extension of the SIS model for the propagation of gonorrhea, which takes into account the fact that
the propagation of the virus is highly non homogeneous among the population --- we refer to the
survey \cite[Section V.A.2]{pastor-satorras_epidemic_2015} from Pastor-Satorras, Castellano, Mieghem
and Vespignani (and more precisely Section 2 therein) for broader context and more details.

In this model the population is divided into $n$ groups and the transmission rates of the disease
between these groups are not equal, leading to a system of coupled ODEs:
\begin{equation}
  \label{finiteSIS}
  \dot{U}_i = (1  - U_i) \sum\limits_{j = 1}^n K_{i,j} \, U_j - \gamma_i U_i,
  \qquad \forall i \in \set{1,2,\ldots,n}
\end{equation}
where ${U}_i$ is the proportion of infected individuals in group $i$ with $U_i(0) \in [0,1]$ for all
$1 \leq i \leq n$, $K = (K_{i,j})_{1 \leq i,j \leq n}$ is a non-negative matrix that represents the
transmission rates of the infection between the different groups, and the non-negative number
$\gamma_i > 0$ is the recovery rate of group $i$. Since the matrix
$K/\gamma = (K_{i,j}/\gamma_j)_{1 \leq i,j \leq n}$ has non-negative entries, we recall it has a
Perron eigenvalue, that is, an eigenvalue $R_0 \in \mathbb{R}_+$ such that all other complex
eigenvalues $\lambda$ of $K/\gamma$ satisfy $|\lambda| \leq R_0$.  The following result is given in
\cite{lajmanovich1976deterministic}.
\begin{enumerate}[1.]
\item There exists a unique solution
  $(U_i(t)\,\colon\, t\geq 0)_{1 \leq i \leq n}$ of Equation
  \eqref{finiteSIS} and $U_i(t) \in [0,1]$ for all
  $t\in \mathbb{R}_+$.
\item If $R_0\leq 1$, $U_i(t)$ converges to $0$ for all
  $1 \leq i \leq n$, so that the disease-free equilibrium
  $(0,0, \ldots, 0)$ is globally stable.
\item If $K$ is irreducible and $R_0>1$, then there exists an endemic
  equilibrium $G^*=(G_i^*)_{1 \leq i \leq n}$ such that for
  $i=1\ldots n$:
  \[
    \lim\limits_{t \to \infty} U_i(t) = G^*_i\in (0, 1),
  \]
  provided that $U(0) \neq (0,0,\ldots,0)$.
\end{enumerate}
  
\medskip
  
Thus, under the assumption that people are connected enough, the epidemic has two possible behaviors
exactly like in the one-dimensional model: \medskip

\begin{quote}{Biotheorem 1, \cite{lajmanovich1976deterministic}}
  Either the epidemic will die out naturally for every possible initial stage of the epidemic, or
  when it is not true and the initial number of infectives of at least one group is nonzero, the
  disease will remain endemic for all the future time. Moreover, the number of infectives and
  susceptibles of each group will approach nonzero constant levels which are independent of the
  initial levels.
\end{quote}

\subsubsection{Towards a generalization}
The epidemiologic models discussed so far assume a large population, possibly made of a few groups
with different behaviours, so that the epidemics is deterministic. At the opposite side of the
modelling spectrum, some probabilistic models of interacting particles may be seen as modelling
epidemics.

In 1974, Harris \cite{harris_contact_1974} introduced the so-called contact process on
$\mathbb{Z}^d$.  The contact process is a continuous-time Markov process often used as a model for
the spread of an infection.  Nodes of the graph represent the individuals of a population. They can
either be infected or healthy. Infected individuals become healthy after an exponential time,
independently of the configuration. Healthy individuals become infected at a rate which is
proportional to the number of infected neighbors. The contact process share a numerous properties
with the multigroup SIS equations: the existence of an upper invariant measure, a disease-free
invariant measure and a monotone coupling \cite{liggett_interacting_1985,liggett_stochastic_1999}.
This proximity is not surprising since Equation~\eqref{finiteSIS} can be obtained from a mean-field
approximation of the contact process \cite[Section~V.A]{pastor-satorras_epidemic_2015}.  Notice that
Equation~\eqref{finiteSIS} can also be obtained as a limit of individual based models, see
\cite{benaim-hirsch-1999}.  \medskip

We refer to \cite{pastor-satorras_epidemic_2015}, and the numerous references therein, for a survey
on epidemic processes in complex networks.  Since social networks are a very large graphs, it is
natural to consider epidemic process on limit of large graphs using the theory developed during the
last two decades, on (i) graphings (which is used to deal with very sparse graphs, namely those with
bounded degree, see \cite{aldous_processes_2007,elek_note_2007,lovasz_large_2012}), or, at the other
extreme, on (ii) graphons (which are comprehensive and flexible objects that define a limit for
dense graphs where the mean degree is of the same order as the number of vertices, see, for example,
\cite{lovasz_large_2012,lovasz_limits_2006}). See
\cite{borgs__2019,borgs_lp_2018,kunszenti-kovacs_measures_2019} for several attempts approach to
limit theory for all kind of graphs.

\medskip

The SIS equation that we propose in the present paper has to be thought as the limit of the
mean-field approximations of the contact processes defined on a convergent sequence of large graphs.
Thus, the solutions take values in an abstract space $\Omega$ (the set of vertices), which can be
interpreted as the set of features of the individuals, the transmission of the disease is given by a
kernel $\kappa$ and the recovery rate by a function $\gamma$ (see Examples \ref{ex:graphonform} and
\ref{ex:graphingform}), see the infinite-dimensional evolution Equation \eqref{eq:SIS} below.

The two main goals of this article are the following:
\begin{itemize}
\item introduce an infinite-dimensional SIS model, generalizing the model developed by Lajmanovitch
  and Yorke (see Equation \eqref{eq:SIS} below), and prove a result similar to
  \cite[Biotheorem~1]{lajmanovich1976deterministic} in that general setting;
\item argue that this general setting is flexible enough to take into account not only the topology
  of the social network, or the disparities between different subgroups of the population, but also
  the effect of vaccination policies (see Section \ref{sec:vacc}), or the effect of lockdown (see
  Section \ref{sec:quarantine}), in the spirit of the policies used to slow down the propagation of
  Covid-19 in 2020.
\end{itemize}

\subsection{The model}
\label{subsec:model}
It is natural to extend the Lajmanovich and Yorke model \eqref{finiteSIS} to a population with an
infinite number of groups. We choose to present this extension in an abstract setting, which will
allows us to include general vaccination and lockdown policies.  We denote by $\Omega$ the set of
the features of the individuals in a given population.  Since $\Omega$ might not be countable, we
shall consider a $\sigma$-field $\mathscr{F}$ on $\Omega$ so that $(\Omega, \mathscr{F})$ is a
measurable space.  We represent the transmission rate from an infinitesimal part of the population
$\mathrm{d}y$ to $x$ by a non-negative kernel $\kappa(x, \mathrm{d}y)$: $\kappa$ is a function from
$\Omega \times \mathscr{F}$ to $\mathbb{R}_+$ such that, for all $A \in \mathscr{F}$, the mapping
$x \mapsto \kappa(x, A)$ is measurable and, for all $x \in \Omega$, the mapping
$A \mapsto \kappa(x, A)$ is a non-negative measure defined on $(\Omega, \mathscr{F})$.  We model the
recovery rate of individuals with feature $x$ by $\gamma(x)$, where $\gamma$ is a non-negative
measurable function defined on $(\Omega, \mathscr{F})$.  The number $1/\gamma(x)$ can be thought as
the typical time of recovery for individuals with feature $x$.  For $x\in \Omega$ and $t\geq 0$, we
denote by $u(t,x)$ the probability for an individual (or the proportion of individuals) with feature
$x$ to be infected at time $t$.  So the intensity of infection attempts on $x$ coming from infected
individuals in $\mathrm{d}y$ is given by $u(t,y)\kappa(x,\mathrm{d}y)$.  Recall that in a SIS model,
the probability for an infection attempt to succeed is proportional to the number of susceptibles
individuals, \textit{i.e.}, those who are not already infected; this explains the term $(1-u(t,x))$
in front of the integral in the next equation.  The evolution equation of the function $u$ for the
SIS model of the probability for being infected is given by the following differential equation (in
infinite dimension):
\begin{equation}
  \label{eq:SIS}
  \left\{
    \begin{aligned}
      \partial_t u(t,x) &= (1-u(t,x)) \int_\Omega u(t,y) \, \kappa(x,\mathrm{d}y) -\gamma(x) u(t,x),
      \quad x \in \Omega, t \in [0, \tau),  \\
      u(0,x) &= u_0(x), \quad x \in \Omega,
    \end{aligned}
  \right.   
\end{equation}
where the measurable function $u_0 \, \colon \, \Omega \to [0,1]$ is the so-called initial condition
and the solution $u$ is defined up to time $\tau \in (0, \infty]$. We shall prove that Equation
\eqref{eq:SIS} is well defined up to $\tau=+\infty $, and we will mainly focus our study on the
long-time behavior of the solutions to this equation and on the study of existence of equilibria.
We refer to Section \ref{sec:discussion} for a discussion on related work, and in particular the
work by Thieme \cite{thieme_global_2011} on spatial SIR model and by Busenberg, Iannelli and Thieme
\cite{busenberg_global_1991} on long-time behavior of an age-structured SIS infection.  \medskip

We check in the following example that the Lajmanovich and Yorke model \eqref{finiteSIS} is a
particular example of \eqref{eq:SIS}.

\begin{example}[Lajmanovich and Yorke model]
  \label{ex:infinite_to_finite}
  Consider a finite set of features, $\Omega = \set{1,2, \ldots, n}$ (with the $\sigma$-field
  $\mathscr{F} = \mathscr{P}(\Omega)$ of all sub-sets of $\Omega$), a finite kernel $\kappa$ and a
  positive recovery rate $\gamma$.  We set for all $i, j \in \Omega$ and $t\geq 0$:
  \[
    K_{i, j} = \kappa(i, \set{j}), \quad \gamma_i = \gamma(i)
    \quad\text{and}\quad U_i(t)=u(t,i),
  \]
  where $u$ is the solution to Equation \eqref{eq:SIS}. Then clearly, the functions $U_i$, for
  $1\leq i\leq n$ solves the finite-dimensional model \eqref{finiteSIS}.
\end{example}

There are two natural extensions of Example \ref{ex:infinite_to_finite} to large bounded degree
graphs and large dense graphs, which is a first approach to model large complex social networks.

\begin{example}[Graph model]\label{ex:graphingform}
  Consider a representation of the social interaction of a population by a simple graph $G$, with
  set of vertices $V(G)=\Omega$ which is at most countable, and set of edges
  $E(G) \subset\Omega \times \Omega$.  For $x\in \Omega$, let
  $\mathscr{N}(x)=\set{y\in G\, \colon \, (x,y) \in E(G)}$ stands for the neighborhood of $x$ in $G$
  and $\mathrm{deg}_G(x)=\mathrm{Card\,}(\mathscr{N}(x))$ for its degree.  If the degree of the
  vertices of $G$ are finite, we may consider a kernel with the following form:
  \begin{equation}\label{eq:graphingform}
    \kappa(x,\mathrm{d}y) = \beta(x) \sum\limits_{z \in \mathscr{N}(x)}
    \theta(y)\delta_{z}(\mathrm{d}y), 
  \end{equation}
  where $\beta$ and $\theta$ are non-negative functions, which represent the susceptibility and the
  infectiousness of the individuals respectively, and $\delta_z$ is the Dirac mass at $z$.  Then
  Equation \eqref{eq:SIS} represents the evolution equation for a SIS model on a graph. The strength
  of the formalism of \eqref{eq:SIS}
  is that one can consider limit of large bounded degree undirected graphs called graphings, see
  Section 18 in \cite{lovasz_large_2012} for the definition of a graphings.
\end{example}

\begin{example}[Graphon form]\label{ex:graphonform}
  One of the initial motivation of this work, was to consider a SIS model on \emph{graphons}, which
  are limit of large dense graphs, see the monograph~\cite{lovasz_large_2012} from Lov\`asz. Recall
  the set of features of the individuals in the population is given by a set $\Omega$.  In this
  approach, the typical form of the transmission kernel $\kappa$ we may consider is:
  \begin{equation}
    \label{eq:graphonform}
    \kappa(x,\mathrm{d}y) = \beta(x) W(x,y) \theta(y) \, \mu(\mathrm{d}y),
  \end{equation}
  where $\beta$ represents the susceptibility and $\theta$ the infectiousness of the individuals;
  $W$ models the graph of the contacts within the population and the quantity $W(x, y)\in[0, 1]$ is
  interpreted as the probability that $x$ and $y$ are connected, or as the density of contacts
  between individuals with features $x$ and $y$; $\mu$ is a probability measure on
  $(\Omega , \mathscr{F})$ and $\mu(\mathrm{d}y)$ represents the infinitesimal proportion of the
  population with feature $y$.  Formally, $\beta$ and $\theta$ are non-negative measurable
  functions, and the function $W \, \colon \Omega \times \Omega \to [0,1]$ is symmetric
  measurable. The quadruple $(\Omega, \mathscr{F}, \mu, W)$ is called a graphon. The degree
  $\mathrm{deg}_W(x)$ of $x\in \Omega$ (\textit{i.e.}  the average number of his contacts) and the
  mean degree $\mathrm{d}_W$ for a graphon $W$ are defined by:
  \begin{equation}
    \label{eq:deg}
    \mathrm{deg}_W(x) = \int_\Omega \!\! W(x,y) \, \mu(\mathrm{d}y)
    \text{ and }
    \mathrm{d}_W = \int_\Omega \!\! \mathrm{deg}_W(x) \, \mu(\mathrm{d}x)
    =  \int_{\Omega^2} \!\! W(x,y) \, \mu(\mathrm{d}y)\, \mu(\mathrm{d}x).
  \end{equation}

  \begin{propenum}
  \item\label{ex:graphonform-c} \textbf{(Constant graphon.)}  One elementary example, is the
    constant graphon, $W=p\in [0, 1]$.  In this case, the degree function is constant, equal to the
    mean degree and thus equal to the parameter $p$.  We recall this constant graphon appears as the
    limit, as $n$ goes to infinity, of Erd\"{o}s-R\'{e}nyi random graphs with $n$ vertices and
    parameter $p$ (that is: independently, for each pair of vertices, there is an edge between those
    two vertices with probability $p$).  If furthermore the functions $\beta$, $\theta$ and $\gamma$
    from~\eqref{eq:SIS} are constant on $\Omega$, then we recover the SIS model \eqref{eq:one-group}
    with $K=p\beta\theta$ and $U(t)=\int_\Omega u(t, x) \mu(\mathrm{d}x)$.

  \item\label{ex:graphonform-sbm} \textbf{(Stochastic block model.)}  The stochastic block models of
    communities corresponds to the case where $W$ is constant by block, \textit{i.e.}  there exists
    a finite partition $(\Omega_i\, \colon \, 1\leq i\leq n)$ of $\Omega$ such that $W$ is constant
    on the blocks $\Omega_i\times \Omega_j$ for all $i,j$, and equal say to $W_{i,j}$.  If
    furthermore, the functions $\beta$, $\theta$ and $\gamma$ from~\eqref{eq:SIS} are also constant
    on the partition, then we recover the Lajmanovich and Yorke model, see \eqref{finiteSIS}, with:
    $K_{i,j}=\beta_i\, W_{i, j} \, \theta_j \, \mu(\Omega_j)$; $\beta_i$, $\theta_i$ and $\gamma_i$
    are the constant values of $\beta$, $\theta$ and $\gamma$ on $\Omega_i$; and
    $U_i(t)=\int_{\Omega_i} u(t, x) \mu(\mathrm{d}x)/ \mu(\Omega_i)$.

  \item\label{ex:graphonform-geom} \textbf{(Geometric graphon.)}  We also mention the geometric
    graphon, where the probability of contact between $x$ and $y$ depends on their relative
    distance. For example, consider the population uniformly spread on the unit circle:
    $\Omega=[0, 2\pi]$ and $\mu(\mathrm{d}x)=\mathrm{d}x/2\pi$. Let $f$ be a measurable non-negative
    function defined on $\mathbb{R}$ which is bounded by 1 and $2\pi$-periodic.  Define the
    corresponding geometric graphon $W_f$ by $W_f(x,y)=f(x-y)$ for $x, y\in \Omega$. In this case,
    the degree of $x\in [0, 1]$ is constant with:
    \[
      \mathrm{deg}_{W_f}(x) = \mathrm{d}_{W_f}=\frac{1}{2\pi} \, \int_{[0, 2\pi]} f(y)\, \mathrm{d}y.
    \]
  \end{propenum}
\end{example}

\subsection{Main assumptions and definition of the reproduction rate}
In order for Equation \eqref{eq:SIS} to make sense, we will need the following assumption. It will
always be in force throughout this paper without supplementary specification.
\begin{assum}\label{Assum0}
  The function $\gamma$ is positive, bounded and the non-negative kernel $\kappa$ is uniformly
  bounded:
  \begin{equation}
    \label{eq:degreebounded}
    \sup\limits_{x \in \Omega} \kappa(x, \Omega) < \infty. 
  \end{equation}
\end{assum}

Assuming the recovery rate $\gamma$ to be bounded is equivalent to require the time of recovery
$1/\gamma$ to be bounded from below by a positive constant. The function $1/\gamma$ is also finite
for all individuals because $\gamma$ is supposed to be positive. It is possible with Assumption
\ref{Assum0} to have individuals with arbitrary large time of recovery, though. Finally, Equation
\eqref{eq:degreebounded} limits the maximal force of infection that can be put upon a susceptible
individual.  \medskip

In Examples \ref{ex:infinite_to_finite}, \ref{ex:graphingform} and \ref{ex:graphonform}, we observe
that the kernel has a density with respect to a reference measure (the counting measure in the first
two examples and the probability measure $\mu$ in the third one).  From an epidemiological point of
view, the reference measure~$\mu$ can be seen as a way to quantify the size of the population and
its sub-groups (defined by a given feature such as sex, spatial coordinates, social condition,
health background, ...). If the measure $\mu$ is finite, then for every measurable set $A$, the
number $\mu(A) / \mu(\Omega)$ is the proportion of individuals in the population whose features
belong to $A$.  We shall consider the case where the density $k$ of $\kappa$ with respect to the
reference measure $\mu $ satisfies some mild integrability condition. We stress that we do not
assume any smoothness condition on the density $k$. By a slight abuse of language, we will also call
the density $k$ a kernel.
\begin{assum}\label{Assum1}
  There exists a finite positive measure $\mu$ on $(\Omega, \mathscr{F})$, a non-negative measurable
  function $k \, \colon \, \Omega \times \Omega \to \mathbb{R}_+$ such that for all $x \in \Omega$,
  $\kappa(x , \mathrm{d}y) = k(x,y) \mu(\mathrm{d}y)$. Besides, there exists $q>1$ such that:
  \begin{equation}\label{eq:strongintegrability}
    \sup\limits_{x \in \Omega} \int_{\Omega} \frac{k(x,y)^q}{\gamma(y)^q} \, \mu(\mathrm{dy}) < \infty.
  \end{equation}
\end{assum}

Notice that since we assume that $\gamma$ is bounded, then Equation \eqref{eq:strongintegrability}
implies the following integrability condition for the kernel $k$:
\begin{equation}\label{eq:weakintegrability}
  \sup\limits_{x \in \Omega} \int_{\Omega} k(x,y)^{q} \,
  \mu(\mathrm{dy})    < \infty. 
\end{equation} 
We shall study in Section~\ref{sec:exple-N} an example which does not satisfy the integrability
condition~\eqref{eq:strongintegrability} nor~\eqref{eq:weakintegrability}.

In addition to Assumption \ref{Assum1}, we will sometimes need the following assumption on the
connectivity of the kernel $k$.
\begin{assum}[Connectivity]\label{Assum_connectivity}
  The kernel $k$ is connected, that is, for all measurable set $A$ such that $\mu(A) > 0$ and
  $\mu(A^c) > 0$, we have that:
  \begin{equation}
    \int_{A \times A^c} k(x,y) \, \mu(\mathrm{d}x) \mu(\mathrm{d}y) > 0.
  \end{equation}
\end{assum}

The sociological interpretation of the connectivity assumption is that we cannot separate the
population into two groups of individuals with no interaction.  Each time Assumptions \ref{Assum1}
or \ref{Assum_connectivity} are used, it will be specified.

\begin{remark}[The finite dimensional case]
  Assumption \ref{Assum1} is automatically satisfied in the finite-dimensional model of Example
  \ref{ex:infinite_to_finite}, where we supposed Assumption \ref{Assum0}. We can indeed take $\mu$
  to be the counting measure and Equation \eqref{eq:strongintegrability} is true because $k$ is
  bounded from above and $\gamma$ is bounded from below by a positive constant as it is positive.
  Notice Assumption \ref{Assum_connectivity} is equivalent to the matrix of transmission rates
  $K = (K_{i,j})_{1 \leq i,j \leq n}$ being irreducible.
\end{remark}

The basic reproduction number of an infection, denoted by $R_0$, has originally been defined as the
number of cases one typical individual generates on average over the course of its infectious
period, in an otherwise uninfected population. This number plays a fundamental role in epidemiology
as it provides a scale to measure how difficult to control an infectious disease is. More
importantly, $R_0$ is often used as a threshold which determines whether the disease will die out
(if~$R_0 <1$) or whether it can invade the population (if~$R_0 > 1$).

In mathematical epidemiology, Diekmann, Heesterbeek and Metz \cite{Diekmann1990} define rigorously
the basic reproduction number for a class of models with heterogeneity in the population. They
propose to consider the next-generation operator which gives the distribution of secondary cases
arising from an infected individual picked randomly according to a certain distribution -- the
population being assumed uninfected otherwise. In our model, under Assumption \ref{Assum1},
following \cite[Equation (4.2)]{Diekmann1990}, we define the next generation operator, denoted by
$T_{k/\gamma}$, as the integral operator:
\begin{equation}
  \label{eq:def-tk/g}
  T_{k/\gamma}(g)(x) =  \int_\Omega \frac{k(x,y)}{\gamma(y)} g(y) \, \mu (\mathrm{d}y)
  \quad\text{for all $x\in \Omega$,}
\end{equation}
which is, thanks to \eqref{eq:strongintegrability}, a bounded positive operator on the space
$\mathscr{L}^\infty(\Omega)$ of bounded real-valued measurable functions defined on $\Omega$.  And
the basic reproduction number is defined by, see \cite[Definition of $R_0$ in Section
2]{Diekmann1990}:
\begin{equation}\label{eq:R0}
  R_0 = r(T_{k / \gamma }),
\end{equation}
where $r$ is the spectral radius, whose exact definition in our general setting will be recalled
below (Equation \eqref{eq:spectral_radius}).  These definitions of the next-generation operator and
the basic reproduction number are consistent with the finite dimensional SIS model given in
\cite{VanDenDriessche}.

\subsection{Long time behavior of solutions to the evolution equation \eqref{eq:SIS}}

We now state our main result concerning solutions of the evolution equation \eqref{eq:SIS}.  Recall
the initial condition of \eqref{eq:SIS}, $u_0$, takes values in $[0, 1]$.
\begin{theorem}\label{th:main_theorem}
  We have the following properties.
  \begin{propenum}
  \item\label{th:main_theoremI} 
    \textbf{(Equation \eqref{eq:SIS} is well defined and $\tau=+\infty$.)}  Under
    Assumption \ref{Assum0}, there exists a unique solution $u$ to
    Equation \eqref{eq:SIS}. This solution is such that, for all
    $(x,t) \in \Omega \times \mathbb{R}_+$, $u(t,x) \in [0,1]$.
  \item\label{th:main_theoremII}
    \textbf{(Disease free equilibrium in the critical and sub-critical case.)}  Assume that
    Assumptions \ref{Assum0} and \ref{Assum1} are in force. Let $R_0$ be defined by
    \eqref{eq:R0}. If $R_0 \leq 1$, then the disease dies out: for all $x \in \Omega$,
    \[
      \lim\limits_{t \to \infty} u(t,x) = 0.
    \]
  \item\label{th:main_theoremIII} 
    \textbf{(Stable endemic equilibrium in the super-critical case.)}  Assume that Assumptions
    \ref{Assum0}, \ref{Assum1} and \ref{Assum_connectivity} are in force. If $R_0>1$, then there
    exists a unique equilibrium $g^* \, \colon \, \Omega\to [0,1]$ with nonzero integral. For all
    initial condition $u_0$ such that its integral is positive:
    \[
      \int_\Omega u_0(x) \, \mu(\mathrm{d}x) >0,
    \]
    the solution $u$ to \eqref{eq:SIS} converges pointwise to $g^*$, \textit{i.e.}, for all
    $x \in \Omega$:
    \[
      \lim\limits_{t \to \infty} u(t,x) = g^*(x).
    \]
    If $u_0=0$ $\mu$-a.e. then the solution $u$ to \eqref{eq:SIS} 
    converges pointwise to~$0$.
  \end{propenum}
\end{theorem}
For property \ref{th:main_theoremI}, see Proposition \ref{prop:deltainv}; property
\ref{th:main_theoremII} is a consequence of Theorems \ref{th:sub} and \ref{th:critical}; and
property \ref{th:main_theoremIII} follows from Corollary \ref{cor:existenceendemic} and Theorem
\ref{th:super}.

\begin{remark}[Uniform convergence]
  The convergence of $u(t, \cdot)$ towards~0 in~(ii) or towards~$g^*$ in~(iii) in Theorem
  \ref{th:main_theorem} is uniform on any measurable subset $A\subset \Omega$ such that
  $\inf_A \gamma>0$, see Theorem~\ref{th:strongconv}. In particular these convergences hold in
  uniform norm if the recovery rate~$\gamma$ is bounded from below.
\end{remark}

\subsection{Modelling vaccination policies, vaccination mechanisms and lockdown}
\subsubsection{Vaccination}
In Section \ref{sec:vacc}, we propose extensions of Equation \eqref{eq:SIS} which take into account
the effect of a vaccination policy. Vaccination confers a direct protection on the targeted
individuals but also acts indirectly on the rest of the population through herd immunity.  However,
all vaccinated individuals will not be totally immune to the disease.
In~\cite{smith_assessment_1984}, Smith, Rodrigues and Fine propose two possible models to explain
vaccine efficacy. In the first model, the vaccine offers complete protection to a portion of the
vaccinated individuals but does not take in the remainder of vaccinated individuals. The second
model supposes that the vaccination confers a partial protection to every vaccinated
individual. In~\cite{shim_distinguishing_2012}, Halloran, Lugini and Struchiner called the former
mechanism the \emph{all-or-nothing} vaccination and the latter one the \emph{leaky} vaccination. We
define below one infinite-dimensional SIS model for each of these two mechanisms.

In order to write the vaccination model, we adapt the one-group SIR models proposed by Shim and
Galvani in \cite{shim_distinguishing_2012} to the one-group SIS model.  Let us denote by
$\eta_\mathtt{v}$ the proportion of vaccinated individuals in the total population, and let
$\eta_\mathtt{u} = 1- \eta_{\mathtt{v}}$.  Let $U_\mathtt{v}$ and $U_\mathtt{u}$ be the proportion
of infected individuals in the vaccinated and unvaccinated population respectively, so that
$\eta_\mathtt{v}U_\mathtt{v} + \eta_\mathtt{u} U_\mathtt{u}$ is the proportion of infected
individuals in the total population.  For both models, we assume that vaccinated individuals who are
nevertheless infected by the disease become less contagious (see
\cite{preziosi_effects_2003,seward_contagiousness_2004} for instance).  We will denote the vaccine
efficacy for infectiousness, that is, the relative reduction of infectiousness for vaccinated
individuals by a parameter $\delta \in [0,1]$.  In what follows, $K$ and $\gamma$ represent the
transmission rate and the recovery rate of the disease as in the model \eqref{eq:one-group} or
\eqref{finiteSIS} and are assumed to be the same for the vaccinated and unvaccinated population.  We
now introduce two models for the so-called vaccine efficacy $e$, see
\cite{haber_measures_1991,halloran_design_1999,shim_distinguishing_2012,smith_assessment_1984} for
discussion on this parameter.  \medskip

In the leaky vaccination, we denote the relative reduction of susceptibility for vaccinated
individual by a parameter $e \in [0,1]$.  Following
\cite[Equations~(1)-(8)]{shim_distinguishing_2012}, with the parameters $\delta$ and $e$
corresponding to $\sigma$ and $\alpha$ in \cite{shim_distinguishing_2012}, the evolution equations
for the leaky vaccination are given by:
\begin{equation}\label{eq:1D_leaky}
  \left\{
    \begin{array}{l}
      \dot{U_\mathtt{v}} = (1 - U_\mathtt{v})(1 - e)
      K ((1-\delta)\eta_\mathtt{v}  U_\mathtt{v} +\eta_\mathtt{u} U_\mathtt{u})
      - \gamma U_\mathtt{v}, \\
      \\
      \dot{U_\mathtt{u}} = (1 - U_\mathtt{u})
      K ( (1-\delta) \eta_\mathtt{v} U_\mathtt{v} + \eta_\mathtt{u} U_\mathtt{u})
      - \gamma U_\mathtt{u}.
    \end{array} 
  \right.
\end{equation}
\medskip

In the all-or-nothing vaccination, we denote the proportion of vaccinated individuals immunized to
the disease (people who can neither contract not transmit the disease) by the parameter
$1-e\in [0, 1]$.  Following \cite[Equations~(13)-(20)]{shim_distinguishing_2012}, the evolution
equations for the all-or-nothing vaccination in the SIS setting are given by:
\begin{equation}\label{eq:1D_all_or_nothing}
  \left\{
    \begin{array}{l}
      \dot{U_\mathtt{v}} = (1 - e - U_\mathtt{v})
      K ((1 - \delta)\eta_\mathtt{v} U_\mathtt{v} +\eta_\mathtt{u}  U_\mathtt{u})
      - \gamma U_\mathtt{v}, \\
      \\
      \dot{U_\mathtt{u}} = (1 - U_\mathtt{u}) K
      ((1 - \delta)\eta_\mathtt{v} U_\mathtt{v} + \eta_{\mathtt{u}}U_\mathtt{u})
      - \gamma U_\mathtt{u}.
    \end{array} 
  \right.
\end{equation}
Since vaccinated individuals that are immunized cannot get the disease, we have
$U_\mathtt{v}(t) \leq 1 - e$ for all $t\in \mathbb{R}_+$.

\begin{remark}
  Notice that, in both models, the unvaccinated population can be viewed as a population inoculated
  with a vaccine of efficacy equal to $0$.
\end{remark}

In Section \ref{sec:vacc}, we derive in Equations \eqref{eq:leaky_multi} and \eqref{eq:aon_multi2}
the analogue of \eqref{eq:1D_leaky} and \eqref{eq:1D_all_or_nothing} in the infinite-dimensional
setting. Those two equations can be seen as a particular case of Equation~\eqref{eq:SIS}.  We also
prove that, as far as the basic reproduction number is concerned the two different vaccination
mechanisms, the all-or-nothing and leaky mechanisms, have the same effect in the infinite
dimensional model, see Proposition \ref{prop:aon=leaky}. This result was already observed in a
one-group model by Shim and Galvani \cite{shim_distinguishing_2012}.  In the case of a perfect
vaccine, where vaccinated people cannot be infected nor infect others, the evolution equation of the
proportion of infected among the non vaccinated population is also given by Equation~\eqref{eq:SIS}
with the kernel $\kappa(x,\mathrm{d}y)$ replaced by $\eta^0(y) \kappa(x,\mathrm{d}y)$ where
$\eta^0(y)$ is the proportion of individuals with feature $x \in \Omega$ which are not vaccinated,
see Equation~\eqref{eq:SIS-vaccinated}. We shall study in a future work the optimal vaccination in
this setting with the basic reproduction number as a criterium to minimize.

\subsubsection{Effect of lockdown policies}
Eventually, we model the effect of lockdown (see Section~\ref{sec:quarantine}) for graphon models
presented in Example \ref{ex:graphonform}, in the spirit of the policies used to slow down the
propagation of Covid-19 in 2020, for example the study in \^Ile de France~\cite{lockdown-IdF}.  In
particular, we prove that a lockdown which bounds the number of contacts of the individuals (this
roughly corresponds to reduce significantly the number of contacts for highly connected groups) is
enough to reduce the basic reproduction number, see Proposition~\ref{cor:lockdown}.  Recall the
definition of the degree $\mathrm{deg}_W(x)$ of $x$ and the mean degree $\mathrm{d}_W$ for a graphon
$W$ defined in~\eqref{eq:deg}. Following Remark~\ref{rem:W=p}, we get that the heterogeneity in the
degree for the graphon model implies larger value of the basic reproduction number. In this
direction, see also~\cite[Section 1.1]{lmdw-2011} on the SIS model from Pastor-Satorras and
Vespignani, where the basic reproduction number increases with the variance of the degrees of the
nodes in a finite graph.

\begin{corollary}
  Consider the SIS model \eqref{eq:SIS} with transmission kernel given in a graphon form
  \eqref{eq:graphonform} (so that $k(x,y)=\beta(x) W(x,y) \theta (y)$). Assume that the
  susceptibility $\beta$, the infectiousness $\theta$ and the recovery rate $\gamma$ are constant
  and positive. The weakest value of the basic reproduction number $R_0$ defined by \eqref{eq:R0}
  among all graphons $W$ with mean degree $\mathrm{d}_W \geq p$ for some threshold $p\in [0, 1]$ is
  obtained for graphons with constant degree equal to $p$ (\textit{i.e.}  graphon $W$ such that
  $\mathrm{deg}_W(x)=p$ for all $x\in \Omega$).
\end{corollary}

We recall from Example \ref{ex:graphonform}~\ref{ex:graphonform-c} and \ref{ex:graphonform-geom},
that the constant graphon and the geometric graphons have constant degree.  Considering a geometric
graphon with (mean) degree $p$, we get that $R_0=\gamma^{-1}\, \beta\theta p$, and for $R_0>1$, we
deduce (directly or from Proposition \ref{prop:g*=cst}), that the equilibrium $g^*$ is constant
equal to $1- R_0^{-1} $ (compare with model \eqref{eq:one-group} with $K=\beta\theta p$).
Furthermore, the example of the geometric graphon with a given mean degree, indicates that, if the
parameters $\beta$, $\theta$ and $\gamma$ are constant, then the contamination distance (or support
of the function $f$, see end of Remark \ref{rem:W=p}) from an infected individual is not relevant
for the value of the basic reproduction number nor for the equilibria.

\subsection{Discussion and related results}
\label{sec:discussion}
The binary dynamic described in \cite[Biotheorem~1]{lajmanovich1976deterministic} has been
established for many other compartmental models and possibly their multigroup version by using
Lyapunov function techniques (see for instance \cite{beretta_global_1986,lin_global_1993}). For a
survey, we refer to Fall, Iggidr, Sallet and Tewa \cite{fall_epidemiological_2007}. In
\cite[Section~6]{hirsch_dynamical_1984} and \cite{smith_cooperative_1986}, Hirsch and Smith proved
the long-time behavior of Equation \eqref{finiteSIS} thanks to their theory of order preserving
systems, thereby giving a completely new perspective to the study of mathematical epidemic models.
Their work greatly inspired Li and Muldowney \cite{li_global_1995} in their important proof of the
global stability of the endemic equilibrium of the SEIR model
(susceptible-exposed-infected-recovered) which was a long-standing conjecture at that time.

In the 70s, models involving transmission rates that depend on the localization of the individuals
\cite{atkinson_deterministic_1976,mollison_possible_1972} or their age \cite{hoppensteadt_age_1974}
were introduced.  Models using localization can be thought of as multigroup models with a continuous
set of groups and therefore lead to differential equations in infinite-dimensional space. In this
setting, results about global stability of the endemic or the disease-free equilibrium have also
been obtained.

In \cite{busenberg_global_1991}, Busenberg, Iannelli and Thieme established the long-time behavior
of an age-structured SIS infection. They proved, thanks to semi-group theory and positive operators
methods, that the system converges to a unique endemic equilibrium if it exists. Otherwise, it
converges to the disease-free equilibrium. In this work the transmission kernel is assumed to be
bounded from above and below by product kernels (see Equation (2.9) therein). This represents a
restriction (see the discussion at the end of \cite{busenberg_global_1991}) as it is not possible to
forbid contacts between some but not all groups. By contrast, in the setting of Example
\ref{ex:graphonform}, it is easy and natural to model the absence of contact between individuals
with feature $x$ and $y$ by imposing that $W(x,y)=0$, without imposing conditions on the probability
of contact between $x$ and other features than $y$.

In \cite{feng_global_2005} Feng, Huang and Castillo-Chavez considered a similar dynamic for a
multigroup age-structured SIS model, but where the endemic equilibrium exists but is not globally
stable.  They assume that the system has a quasi-irreducibility property (see Definition~3.1
therein) which is a weaker assumption than Assumption~\ref{Assum_connectivity}, but impose bounds on
the transmission kernel.

In \cite{thieme_spectral_2009}, Thieme also used an operator approach to study a SIR model with
variable susceptibility (see Section 4 therein).  In particular, he studied the close relation
between the spectral bound of the operator $T_k -\gamma$ and of the basic reproduction number $R_0$
which is the spectral radius of the operator $T_{k/\gamma}$.  In \cite{thieme_global_2011}, Thieme
analyzed a space-structured SIR model with birth.  In this model, the incidence term, \textit{i.e.}
the equivalent of $(1-u(t,x)) u(t,y) \kappa(x, \mathrm{d}y)$ in Equation~\eqref{eq:SIS}, is replaced
by a non bilinear term $f(x,y,1-u(t,x), u(t,y)) \, \mathrm{d}y$, where the function $f$ is
continuous, locally Lipschitz continuous and increasing in its third and fourth argument. Imposing
also that the recovery rate~$\gamma$ is bounded away from 0, he proved an analogue of
\cite[Biotheorem~1]{lajmanovich1976deterministic} (see Theorems 7.1, 8.2, 9.1 and 12.1 therein)
using Lyapunov functions. Part of those results would not hold in general if $\inf \gamma=0$. In
contrast to those works, we consider very few regularity assumptions on the parameters, and in
particular allow that $\inf \gamma=0$.

\medskip

The principal tools we use to prove Theorem~\ref{th:main_theorem}, see also the key
Lemma~\ref{lem:schae} and Proposition~\ref{prop:F}, can be summarized as follows.
\begin{description}
\item[Cooperative systems] The function $g\mapsto F(g)=(1-g) T_\kappa(g) -\gamma g$ is
  \emph{cooperative} (see Definition \ref{def:QuasiMonotony} and Remarks \ref{rem:DiffCoop1} and
  \ref{rem:DiffCoop2}), which implies that the solution of \eqref{eq:SIS} are well defined and the
  corresponding dynamical system is order preserving.  For an approach based on cooperation (or
  quasi-monotonicity) and monotone dynamical system, see \cite{hirsch_dynamical_1984,
    smith_cooperative_1986, smith-thieme-1991, smith_1995, benaim-hirsch-1999,hirsch-smith2005}.

\item[Positive operators] Under Assumption \ref{Assum1} and Equation \eqref{eq:strongintegrability},
  the integral operator $T_{k/\gamma}$ can be seen as an Hille-Tamarkin operator on $L^p(\mu)$ with
  the corresponding compactness property see \cite[Theorem 41.6]{zaanen}. Then the positivity of the
  operator $T_{k/\gamma}$ allows to use Krein-Rutman theorem to get that its spectral radius is an
  eigenvalue with a non-negative eigen-function. This argument has been widely used, see for example
  \cite{busenberg_global_1991} (where the operator is of rank one, and thus is compact) and also
  \cite{thieme_spectral_2009, thieme_global_2011}.

\item[Connectivity] Under Assumption \ref{Assum_connectivity} on the connectivity of kernel $k$
  (which in finite dimension corresponds to the irreducibility of non-negative matrices and is
  related to the Perron-Frobenius theorem), we can consider the unique corresponding eigenvector,
  thanks to the Perron-Jentzsch theorem (see \cite[Theorem V.6.6]{schaefer_banach_1974} or
  \cite[Theorem 5.2]{grobler-1995}). This eigenvector is an essential tool to study the long-time
  behavior of the solution to Equation \eqref{eq:SIS} in the super-critical regime.  In finite
  dimension, see \cite{lajmanovich1976deterministic}, where the matrix $K$ from \eqref{finiteSIS} is
  assumed to be irreducible, or \cite{benaim-hirsch-1999} for a more general finite-dimensional
  model. In infinite dimension, see \cite{feng_global_2005} for a weaker quasi-irreducibility
  condition.
\end{description}
Finally let us remark that we do not use the standard tool of Lyapunov functions, in contrast with
many previous works, see for example~\cite{beretta_global_1986,lin_global_1993,thieme_global_2011}.

\subsection{Structure of the paper}
In Section~\ref{sec:Fproperty} we construct the semi-flow associated to the infinite dimensional SIS
model \eqref{eq:SIS}, and prove its main regularity and monotonicity properties. We introduce in
Section~\ref{sec:dense} some important tools of spectral analysis in Banach lattices. This allows
us to define in Section~\ref{subsec:R0} the basic reproduction number $R_0$. The convergence of the
system towards an equilibrium is established in Section~\ref{sec:equilibrium}. In Section
\ref{sec:vacc}, we take into account the effect of a vaccination policies on the propagation of the
disease. Eventually, in Section~\ref{sec:quarantine}, we model the impact of lockdown policies on
the propagation of the disease when $\kappa$ takes the graphon form of Example \ref{ex:graphonform}.


\section{Model analysis}\label{sec:Fproperty}
\subsection{Preamble}\label{subsec:preamble}
In this paragraph, we recall some definitions of functional analysis. Most of them can be found in
\cite{deimling2010nonlinear}. Let $(X,\norm{\cdot})$ be a Banach space. The topological dual
$X^\star$ of $X$ is the space of all bounded linear forms and we use the notation
$\braket{x^\star, x}$ for the value of an element $x^\star \in X^\star$ at $x \in X$. We consider
$K$ a proper cone on $X$, \textit{i.e.} a closed convex subset of $X$ such that
$\lambda K \subset K$ for all $\lambda \geq 0$ and $K \cap (-K) = \set{0}$. The proper cone $K$
defines a partial ordering $\leq$ on $X$: $x\leq y$ if $y - x \in K$. It is said to be reproducing
if $K - K = X$ (any element $x \in X$ can be expressed as a difference of elements of $K$). The dual
cone of $K$ is the set $K^\star \subset X^\star$ consisting of all $x^\star$ such that
$\braket{x^\star,x} \geq 0$ for all $x \in K$. If the proper cone $K$ is reproducing then the set
$K^\star$ is a proper cone (see beginning of \cite[Section 19.2]{deimling2010nonlinear}).

We denote by $\mathcal{L}(X)$ the space of bounded linear operators from $X$ to $X$.  The operator
norm of a bounded operator $A \in \mathcal{L}(X)$ is given by:
\[
  \norm{A} = \sup \set{ \norm{Ax} \, \colon \, x \in X, \, \norm{x} \leq 1}.
\]
The topology associated to $\norm{\cdot}$ in $\mathcal{L}(X)$ is called the uniform operator
topology. A linear bounded operator $A \in \mathcal{L}(X)$ is said to be positive (with respect to
the proper cone $K$) if $A K \subset K$.

Let $F$ be a function defined on an open domain $D \subset X$ and taking values in $X$. The function
$F$ is said to be Fr\'echet differentiable at $x \in D$, if there exists a bounded linear operator
$\mathcal{D}F[x]$ such that:
\[ 
  \lim\limits_{y \to 0} \norm{F(x + y) - F(x) - \mathcal{D}F[x](y)}/\norm{y} = 0.
\]
The operator $\mathcal{D}F[x]$ is called the Fr\'echet derivative of $F$ at point $x$.

We define the cooperativeness property which is related to the definition of quasimonotony firstly
introduced by Volkmann \cite{volkmann} for abstract operators.
\begin{definition}[Cooperative function]
  \label{def:QuasiMonotony}
  Let $D_1, D_2 \subset X$. A function $F \, \colon \, X \to X$ is said to be cooperative on $D_1 \times D_2$ (with respect to $K$)
  if, for all $(x,y) \in D_1 \times D_2$
  such that $x\leq y$ and for all $z^\star \in K^\star$, we have the following property: 
  \begin{equation}
    \label{eq:Quasimonotone}
    \braket{z^\star,x - y} = 0 \implies \braket{z^\star,F(x) - F(y)} \leq 0.
  \end{equation}
\end{definition}
We shall mainly consider the cases $D_1 = X$ or $D_2 = X$.

\begin{remark}\label{rem:DiffCoop1}
  For a better understanding of the cooperativeness property, let us examine the finite dimensional
  case. Let $d \geq 2$, $ X = \mathbb{R}^d$ and $K=\mathbb{R}_+^d$. Then, for a smooth function
  $F=(F_1, F_2, \ldots, F_d)$, it is easy to see that $F$ is cooperative on $X \times X$ with
  respect to $K$ if and only if:
\begin{equation}
  \label{eq:cooperative}
  \frac{\partial F_j}{\partial x_i}(x) \geq 0 \quad\text{for all $x \in
    \mathbb{R}^d$ and all $i \neq j$}.
\end{equation}
We recover the definition of cooperativeness introduced by Hirsch
\cite{hirsch_dynamical_1984}. Suppose the vector $x$ represent the utilities of a group of agents
$\set{1,2, \ldots d}$ and $F$ is the dynamic of the system, \textit{i.e.}, $\dot{x} = F(x)$. Then,
the higher the utilities of agents $j \neq i$ are, the more the situation is beneficial for the
agent $i$ as it increases the value of the time derivative of $x_i$. For this reason, the function
$F$ satisfying \eqref{eq:cooperative} is called cooperative.
\end{remark}

We extend the differential version of cooperativeness of Remark \ref{rem:DiffCoop1} to infinite
dimension in the next remark.

\begin{remark}\label{rem:DiffCoop2}
  Let $X$ be a Banach space, $D$ an open domain and $F \, \colon \, X \to X$ be a Fr\'echet
  differentiable function. Assume that $F$ is cooperative on $D \times X$. Let
  $(x,z) \in D \times K$ and let $z^\star \in K^\star$ such that $\braket{z^\star, z} = 0$. Since
  $F$ is cooperative on $D \times X$, we have:
  \[
    \braket{z^\star, (F(x + \lambda z) - F(x))/\lambda} \geq 0,
  \]
  for all $\lambda>0$. Letting $\lambda$ go to $0$, we obtain the following inequality:
\begin{equation}\label{eq:diff_cooperativeness}
\braket{z^\star, \mathcal{D}F[x](z)} \geq 0.
\end{equation}

Using path integrals in Banach space, we can prove the reverse implication in the case
$D=X$. Indeed, for all $x,y \in X$ and all $z^\star \in X^\star$, we have:
\begin{equation}\label{eq:path_integral}
  \braket{z^\star, F(x) - F(y)} =-  \int_0^1
  \braket{z^\star, \mathcal{D}F[(1 - \lambda) x + \lambda y](y - x)} \, \mathrm{d} \lambda.
\end{equation}
Assume \eqref{eq:diff_cooperativeness} holds for $z^\star \in K^\star$ and $z\in K$.  Then, if
$x \leq y$, $z^\star \in K^\star$ and $\braket{z^\star, y-x} = 0$, we get that
$\braket{z^\star, F(x) - F(y)}$ is non-positive thanks to Equation
\eqref{eq:diff_cooperativeness}. Thus the function $F$ is cooperative.
\end{remark}

Ordinary differential equations (ODEs) driven by cooperative vector fields enjoy a number of nice
properties that we now review.  Let us first recall a few definitions and classical properties of
ODEs. Let $a>0$. We consider a function $G \, \colon \, [0,a) \times X \to X$. We suppose that $G$
is locally Lipschitz in the second variable, that is: for all $(t,x) \in [0,a) \times X$, there
exist $\eta = \eta(t,x)>0$, $L = L(t,x) > 0$ and a neighborhood $U_x$ of $x$ such that
$\norm{G(s, y) - G(s, z)} \leq L \norm{y - z}$ for all $s \in [0,a) \cap [t, t+\eta]$ and
$y, z \in U_x$. With this assumption over $G$, the Picard- Lindel\"{o}f theorem ensures the
existence of $0<b \leq a$ and a continuously differentiable function $y$ from $J=[0,b)$ to $X$ which
is the unique solution of the Cauchy problem:
\begin{equation}
  \label{eq:Cauchy}
  \left\{
    \begin{array}{l}
      y'(t) = G(t, y(t)) \quad t \in J, \\
      \\
      y(0) = y_0,
    \end{array} 
  \right. 
\end{equation}
where $y_0 \in X$ is the so-called initial condition (see \cite[Section~1.1]{deimlingordinary}). A
solution $y$ defined on an interval $[0,b)$ is said to be maximal if there is no solution of
Equation \eqref{eq:Cauchy} defined on $[0,c)$ with $c>b$. A solution is said to be global if it is
defined on $[0,a)$.

Global existence, existence and theorems on differential inequalities are intimately connected with
the flow invariance of certain subsets in the domain of $G$, \textit{i.e.}, the question whether
every solution starting in $D$ remains in $D$ as long as it exists. We recall the definition of flow
invariance given in \cite[Section~5]{deimlingordinary}.
\begin{definition}[Forward invariance]\label{def:forwardinvariant}
  A set $D \subset X$ is said to be forward invariant with respect to $G$ if the maximal solution
  $(y, J)$ of the Cauchy problem \eqref{eq:Cauchy} takes values in $D$ for $t\in J$ provided that
  $y_0 \in D$.
\end{definition}

In most applications, the set $D$ owns a structure which make the forward invariance easier to
show. For instance, when $D$ is the translation of a cone, the forward invariance is implied by the
following condition, which is proved in Section \ref{sec:proofs}.
\begin{theorem}\label{th:Inv}
  Let $G\, \colon \, [0,a ) \times X \to X$ be locally Lipschitz in the second variable. Let $K$ be
  a proper cone of $X$ with non-empty interior and $y \in X$. If for all
  $(x,t) \in \partial K \times [0,a)$ and for all $x^\star \in K^\star$ such that
  $\braket{x^\star,x} = 0$, we have: $\braket{x^\star,G(t,y+x)} \geq 0$, then $y+K$ is forward
  invariant with respect to $G$.
\end{theorem}

We will mainly use this result through the following corollary which roughly asserts that
cooperative flow preserve the order.

\begin{corollary}[Comparison Theorem]\label{cor:Comparison}
  Let $K$ be a proper cone of $X$ with non-empty interior. Denote by $\leq$ the corresponding
  partial order. Let $F \, \colon \, X \to X$ be locally Lipschitz, $D_1, D_2 \subset X$, $a > 0$,
  and let $u \, \colon \, [0,a) \to D_1$ and $v \, \colon \, [0,a) \to D_2$ be $\mathcal{C}^1$ paths
  such that $u(0) \leq v(0)$. We suppose that $F$ is cooperative on $D_1 \times X$ or on
  $X \times D_2$, and that:
  \begin{equation}\label{eq:Comparison}
    u'(t) - F(u(t)) \leq v'(t) - F(v(t)) \quad \forall t \in [0,a).
  \end{equation}
  Then, we have that: $u(t) \leq v(t)$ for all $t \in [0,a)$.
\end{corollary}

Corollary \ref{cor:Comparison} is in the spirit of \cite[Theorem~5.2]{deimlingordinary}. For the
sake of completeness, a proof is given in Section \ref{sec:proofs}.

\subsection{Notations}
\label{subsec:Not}
In this section, we will work in the Banach space $\mathscr{L}^\infty(\Omega)$ of measurable bounded
real-valued functions defined on $\Omega$ equipped with the supremum norm $\norm{\cdot}$. We shall
write $\mathscr{L}^\infty$ when there is no ambiguity on the underlying space. The set:
\begin{equation}
  \label{eq:coneK}
  \mathscr{L}^\infty_+ = \set{ f \in \mathscr{L}^\infty \, \colon \, f(x) \geq 0 \quad \forall x \in \Omega},
\end{equation}
is a proper cone in $\mathscr{L}^\infty$ with non-empty interior. The order defined by this proper
cone is the usual order: $g \leq h$ if $g(x) \leq h(x)$ for all $x \in \Omega$.

We denote by $\mathscr{L}^{\infty, \star}$, the topological dual of $\mathscr{L}^\infty$. It can be
identified as the space of bounded and finitely additive signed measures on $\Omega$ equipped with
the total variation norm (see \cite[Section~2]{yosida_finitely_1952}). Since $\mathscr{L}^\infty_+$
is reproducing, the dual cone $\mathscr{L}^{\infty, \star}_+$ is a proper cone. It consists of the
continuous linear positive forms on $\mathscr{L}^\infty$.

\medskip

Let $\kappa$ be a non-negative kernel on $\mathscr{L}^\infty$ (endowed with its Borel
$\sigma$-field) satisfying Assumption \ref{Assum0}.  We denote by $T_\kappa$ the operator:
\begin{align}
\label{eq:def-T}
  T_\kappa \quad \colon \quad \mathscr{L}^\infty &\to \mathscr{L}^\infty \\
  \nonumber
  g &\mapsto \left( x \mapsto \int_\Omega g(y) \, \kappa(x, \mathrm{d}y) \right).
\end{align}
According to Assumption \ref{Assum0}, the operator $T_\kappa$ is a bounded linear operator with:
\begin{equation}
  \norm{T_\kappa} = \sup\limits_{x \in \Omega} \kappa(x, \Omega) < \infty.
\end{equation}
Since, for all $x \in \Omega$, $\kappa(x, \mathrm{d}y)$ is a positive measure, the operator
$T_\kappa$ is moreover positive. We also define the function $F$ from $\mathscr{L}^\infty$ to
$\mathscr{L}^\infty$ by:
\begin{equation}
  \label{eq:DefF}
  F(g) = (1 - g) T_\kappa(g) - \gamma g.
\end{equation}
Then, Equation \eqref{eq:SIS} may be rewritten as an ODE in the Banach space
$(\mathscr{L}^\infty, \norm{\cdot})$:
\begin{equation}
\label{eq:SISODE}
  \left\{
    \begin{array}{l}
      \partial_t u = F(u), \quad t \in [0, \tau) \\
      \\
      u(0, \cdot) = u_0,
    \end{array} 
  \right. 
\end{equation}
where $u_0 \in \mathscr{L}^\infty$ and $\tau \in (0, \infty]$. Let $\Delta$ be the set of
non-negative functions bounded by~1:
\begin{equation}
  \label{eq:Delta}
  \Delta = \set{ f \in \mathscr{L}^\infty \, \colon \, 0 \leq f \leq 1 }.
\end{equation}
Since the solution $u(t,x)$ of Equation \eqref{eq:SIS} defines the proportion of $x$-type
individuals being infected at time $t$, it should remain below $1$ and above $0$. Hence, for
\eqref{eq:SIS} to make a biological sense, the initial condition should belong to $\Delta$ and the
solution (if it exists) should remain in $\Delta$ (that is $\Delta$ is forward invariant with
respect to $F$). This will be checked in Proposition \ref{prop:deltainv}.

\subsection{Properties of the vector field}
\label{subsec:Fproposition}
Recall that Assumption \ref{Assum0} is in force. The main results of this section are gathered in
the following proposition.
\begin{proposition}[Properties of $F$]
  \label{prop:F}
  The function $F$ defined in \eqref{eq:DefF} has the following properties.
  \begin{propenum}
  \item \label{prop:Fsmooth} $F$ is of class $\mathcal{C}^\infty$ on $\mathscr{L}^\infty$.
  \item \label{prop:FLipschitz} $F$ and its repeated derivatives are bounded on bounded sets.
  \item \label{prop:Fproducttopo} $F$ is continuous on $\Delta$ with respect to the topology of pointwise convergence.
  \item \label{prop:Fcoop} $F$ is cooperative on $(1-\mathscr{L}^\infty_+) \times \mathscr{L}^\infty$ where:
    \begin{equation}\label{eq:cone1-K}
      1-\mathscr{L}^\infty_+ = \set{ g \in \mathscr{L}^\infty \, \colon \, g \leq 1 }.
    \end{equation}
  \end{propenum}
\end{proposition}
\begin{proof}
  The bilinear map $(g,h) \mapsto gh$ and the linear maps $g \mapsto \gamma g$ and
  $g \mapsto T_\kappa(g)$ are bounded on $\mathscr{L}^\infty$ (hence, smooth as they are
  linear). Since the function $F$ is a sum of compositions of the previous maps, properties
  \ref{prop:Fsmooth} and \ref{prop:FLipschitz} are proved.

  \medskip

  Now, we prove property \ref{prop:Fproducttopo}.  Let $(g_n, \; n \in \mathbb{N})$ be a sequence of
  functions in $\Delta$ converging pointwise to $g \in \Delta$. Let $x \in \Omega$. The functions
  $g_n$ are dominated by the function equal to $1$ everywhere. The latter is integrable with respect
  to the measure $\kappa(x, \mathrm{d}y)$ since $\kappa(x, \Omega) < \infty$ according to
  \eqref{eq:degreebounded}. Therefore, we can apply the dominated convergence theorem and obtain:
  \[
    \lim\limits_{n \to \infty} \int_\Omega g_n(y) \kappa(x, \mathrm{d}y) = \int_\Omega g(y)
    \kappa(x, \mathrm{d}y).
  \]
  Thus, the operator $T_\kappa$ is continuous on $\Delta$ with respect to the pointwise convergence
  topology. The maps $(h_1,h_2) \mapsto h_1 h_2$ and $(h_1,h_2) \mapsto h_1 + h_2$ are also
  continuous with respect to the pointwise convergence topology. Hence, property
  \ref{prop:Fproducttopo} is proved since $F$ is a composition of these functions.

  \medskip

  Finally, let us prove property \ref{prop:Fcoop}.  Let $g, h \in \mathscr{L}^\infty$ such that
  $g \leq 1$ and $g \leq h$, and let $\nu \in \mathscr{L}^{\infty, \star}_+$ such that
  $\braket{\nu,g - h} = 0$. We have:
  \begin{align*}
    \braket{\nu, F(g)  - F(h)} &= \braket{\nu,(1-g) T_\kappa(g - h) + (h - g) (T_\kappa(h) + \gamma )} \\
                               &= \braket{\nu,(1-g) T_\kappa(g - h)},
  \end{align*}
  where we used Lemma \ref{lem:Stability} below (with $g$ replaced by $h-g$ and $h$ by
  $T_\kappa(h) + \gamma$) in order to get that $\braket{\nu,(h - g) (T_\kappa(h) + \gamma )}$ is
  equal to $0$. Since $T_\kappa$ is a positive operator and $g \leq 1$, the function
  $(1-g) T_\kappa(g - h)$ is non-positive. The number $\braket{\nu,(1-g) T_\kappa(g - h)}$ is also
  non-positive because $\nu \in \mathscr{L}^{\infty, \star}_+$. Hence, we get that
  $\braket{\nu,F(g) - F(h)} \leq 0$. This ends the proof thanks to Definition
  \ref{def:QuasiMonotony}.
\end{proof}

Property \ref{prop:Fcoop} of Proposition \ref{prop:F} was proved using the following lemma.
\begin{lemma}
  \label{lem:Stability}
  Let $g \in \mathscr{L}^\infty_+$ and $\nu \in \mathscr{L}^{\infty, \star}_+$ such that
  $\braket{\nu,g} =0$. Then, for all $h \in \mathscr{L}^\infty$, we have $\braket{\nu,h g} =0$.
\end{lemma}
\begin{proof}
  Let $g \in \mathscr{L}^\infty_+$ and $\nu \in \mathscr{L}^{\infty, \star}_+$ such that
  $\braket{\nu, g} = 0$. Since $g$ is everywhere non-negative, we have:
  \[
    -\norm{h} \, g  \leq h g \leq \norm{h} \, g.
  \]
  Since $\nu \in \mathscr{L}^{\infty, \star}_+$, the previous inequalities give:
  \[ 
    -\norm{h} \, \braket{\nu,g}  \leq \braket{\nu,h g} \leq \norm{h} \, \braket{\nu,g}.
  \]
  By assumption, $\braket{\nu,g}$ is equal to $0$. Hence, the lemma is proved.
\end{proof}

\subsection{Properties of the ODE semi-flow}
The aim of this subsection is to define a semi-flow associated to Equation \eqref{eq:SIS} and to
study its main properties. Proposition \ref{prop:F}~\ref{prop:FLipschitz} enables to apply the
Picard-Lindel\"{o}f theorem and show the existence of local solutions of in $\mathscr{L}^\infty$ of
Equation \eqref{eq:SIS}. We can actually prove a stronger result. Recall that
$\Delta = \set{ f \in \mathscr{L}^\infty \, \colon \, 0 \leq f \leq 1 }$.
\begin{proposition}
\label{prop:deltainv} 
Let $F$ defined by \eqref{eq:DefF}.
\begin{propenum}
  \item \label{prop:deltainvI} The domain $\Delta$ is forward invariant with respect to $F$.
  \item \label{prop:deltainvII} Maximal solutions of Equation \eqref{eq:SISODE} such that
    $u_0 \in \Delta$ are global, \textit{i.e.}, they are defined on $\mathbb{R}_+$.
\end{propenum}
\end{proposition}
\begin{proof}
  We first prove property \ref{prop:deltainvI}. The domain $\Delta$ is the intersection of the cone
  $\mathscr{L}^\infty_+$ defined by \eqref{eq:coneK} and the set $1-\mathscr{L}^\infty_+$ defined by
  \eqref{eq:cone1-K}. So, it is sufficient to prove that each of them is forward invariant with
  respect to $F$. Let $g \in \partial \mathscr{L}^\infty_+$ and let
  $\nu \in \mathscr{L}^{\infty, \star}_+$ such that $\braket{\nu,g} = 0$.
    
  The function $F$ being cooperative on $(1-\mathscr{L}^\infty_+) \times \mathscr{L}^\infty$
  according to Proposition \ref{prop:F}~\ref{prop:Fcoop}, the following inequality holds:
  \[
    \braket{\nu, F(0) - F(g)} \leq 0.
  \]
  We deduce that $\braket{\nu, F(g)} \geq 0$ because $F(0)=0$. Since $\mathscr{L}^\infty_+$ is a
  proper cone with non-empty interior, we can apply Theorem \ref{th:Inv} with
  $G(t, \cdot) = F(\cdot)$, $K=\mathscr{L}^\infty_+$, $x=g$, $x^\star = \nu$ and $y=0$ in order to
  get that $\mathscr{L}^\infty_+$ is forward invariant with respect to $F$.
    
  The function $F$ being cooperative on $(1-\mathscr{L}^\infty_+) \times \mathscr{L}^\infty$
  according to Proposition \ref{prop:F}~\ref{prop:Fcoop}, the following inequality holds:
  \[
    \braket{\nu,F(1 - g)- F(1)} \leq 0.
  \]
  Since $F(1) = - \gamma \leq 0$, we get that $\braket{\nu,F(1 - g)} \leq 0$. By using Theorem
  \ref{th:Inv} with $G(t, \cdot) = F(\cdot)$, $K=-\mathscr{L}^\infty_+$, $x =-g$, $y=1$ and
  $x^\star = - \nu$, we obtain that $1-\mathscr{L}^\infty_+$ is forward invariant with respect to
  $F$. This ends the proof of property \ref{prop:deltainvI}.

  \medskip

  Now we prove property \ref{prop:deltainvII}.  Let $(y , [0, \tau))$ be a solution of Equation
  \eqref{eq:SIS} with $y(0) \in \Delta$.  Assume that $\tau$ is a positive finite number. Property
  \ref{prop:deltainvI} asserts that $y(t) \in \Delta$, for all $0\leq t< \tau$. Since $F$ is bounded
  on $\Delta$ (see Proposition \ref{prop:F}~\ref{prop:FLipschitz}), $s \mapsto F(y(s))$ is
  integrable and:
  \[
    \lim\limits_{t \to \tau^-} y(t) = y(0) + \lim\limits_{t \to \tau^-} \int_0^t F(y(s)) \,
    \mathrm{d}s = y(0) + \int_0^\tau F(y(s)) \, \mathrm{d}s.
  \] 
  The solution $y$ can be extended up to its right boundary, \textit{i.e.}, on $[0, \tau]$. This
  shows that $y$ is not maximal. We deduce that the maximal solution is defined on $\mathbb{R}_+$.
\end{proof}

Thanks to Proposition \ref{prop:deltainv}, it is possible to define the semi-flow associated to the
autonomous differential equation \eqref{eq:SIS} on $\Delta$, \textit{i.e.}, the unique function
$\phi \, \colon \, \mathbb{R}_+ \times \Delta \to \Delta$ solution of:
\begin{equation}
\label{eq:semiflow}
  \left\{
    \begin{array}{l}
      \partial_t \phi(t, g) = F(\phi(t,g)), \\
      \\
      \phi(0,g) = g.
    \end{array} 
  \right. 
\end{equation}
It satisfies the semi-group property, that is, for all $g \in \Delta$ and for all
$t,s \in \mathbb{R}_+$, we have:
  \[
    \phi(t+s, g) = \phi(t, \phi(s,g)).
  \]

  The result below is a fundamental property about the semi-flow of the SIS model. It express the
  intuitive idea that if an epidemics is worse everywhere compared to a reference state, it will
  remain worse compared to the evolution of this reference state in the future.

\begin{proposition}[Order-preserving flow]
  \label{prop:Ordpres}
  If $0 \leq g \leq h \leq 1$, then we have $\phi(t, g) \leq \phi(t, h)$ for all
  $t\in \mathbb{R}_+$.
\end{proposition}
\begin{proof}
  Since $\partial_t \phi(t, g) - F(\phi(t, g)) =0$ and $\partial_t \phi(t, h) - F(\phi(t, h)) =0$,
  the inequality \eqref{eq:Comparison} is satisfied on $\mathbb{R}_+$ for the paths
  $u \, \colon \, t \mapsto \phi(t,g)$ and $v \, \colon \, t \mapsto \phi(t,h)$. By assumption, we
  have also that $g = \phi(0,g) \leq \phi(0,h) = h$.  Furthermore $F$ is locally Lipschitz (see
  Proposition \ref{prop:F}~\ref{prop:FLipschitz}) and cooperative on
  $(1-\mathscr{L}^\infty_+) \times \mathscr{L}^\infty$ (see Proposition
  \ref{prop:F}~\ref{prop:Fcoop}) and thus on $\Delta \times \mathscr{L}^\infty$.  Hence, we can
  apply Corollary \ref{cor:Comparison} with $u(t) = \phi(t,g)$ and $v(t) = \phi(t,h)$ to obtain that
  $\phi(t, g) \leq \phi(t, h)$ for all $t\in \mathbb{R}_+$.
\end{proof}
As a consequence of the previous proposition, we have the following result.
\begin{corollary}[Local Monotony implies Global Monotony]
  \label{cor:nodecreas}
  Let $g \in \Delta$.  Suppose that there exist $0 \leq a<b$ such that, for all $t \in [a,b)$, the
  inequality $\phi(a, g) \leq \phi(t,g)$ (resp. $\phi(a, g) \geq \phi(t,g)$) holds. Then,
  $t \mapsto \phi(t, g)$ is non-decreasing (resp. non-increasing) on $[a, \infty)$.
\end{corollary}
\begin{proof}
  It is sufficient to show that $t \mapsto \phi(t, g)$ is non-decreasing on all subintervals of
  $[a, \infty)$ whose lengths are bounded from above by $b-a$. Let $t>s\geq a$ such that
  $t - s < b-a$. By assumption, we have: $\phi(a, g) \leq \phi(a + t - s, g)$.  Thus, Proposition
  \ref{prop:Ordpres} gives:
  \[
    \phi(s - a, \phi(a, g)) \leq \phi(s - a, \phi(a + t - s, g)).
  \]
  By the semi-group property of the semi-flow, this implies that $\phi(s, g) \leq \phi(t, g)$.
\end{proof}

\begin{proposition}\label{prop:propC}
  Let $g \in \Delta$. The path $t \mapsto \phi(t,g)$ is non-decreasing (resp. non-increasing) if and
  only if $F(g) \geq 0$ (resp. $F(g) \leq 0$).
\end{proposition}
\begin{proof}
  Let $g\in \Delta$. Suppose $F(g) \geq 0$. Let $h \geq g$ and
  $\nu \in \mathscr{L}^{\infty, \star}_+$. According to Proposition \ref{prop:F}~\ref{prop:Fcoop},
  if $\braket{\nu, h -g} = 0$, then $\braket{\nu, F(h) - F(g)} \geq 0$. Since $\nu$ is a positive
  linear form, it follows that $\braket{\nu, F(g)}\geq 0$ and then $\braket{\nu, F(h)}\geq
  0$. Applying Theorem \ref{th:Inv} with $G(t, \cdot) = F(\cdot)$, $K = \mathscr{L}^\infty_+$,
  $x = h-g$, $y = g$ and $x^\star = \nu$ (and assuming that $h-g\in \partial \mathscr{L}^\infty_+$),
  we get that $g + \mathscr{L}^\infty_+$ is forward invariant with respect to $F$. This means that
  for all $t\geq 0$, the following inequality holds:
  \[
    g = \phi(0, g) \leq \phi(t, g). 
  \]
  Thus, $t \mapsto \phi(t,g)$ is non-decreasing according to Corollary \ref{cor:nodecreas}.

  \medskip

  Now, suppose that $t \mapsto \phi(t,g)$ is non-decreasing. Then, for all $t>0$, the function
  $(\phi(t, g) - g) / t$ belongs to $ \mathscr{L}^\infty_+$. Since $ \mathscr{L}^\infty_+$ is
  closed, it follows that $F(g)=\lim_{t\rightarrow 0+} (\phi(t, g) - g) / t$ belongs also to
  $ \mathscr{L}^\infty_+$.  \medskip

  The equivalence between $F(g) \leq 0$ and the fact that $t \mapsto \phi(t,g)$ is non-increasing is
  proved the same way.
\end{proof}

Now we give some results about the regularity of the semi-flow.
\begin{proposition}[Flow regularity]
  \label{prop:Flow}
  Let $\phi \, \colon \, \mathbb{R}_+ \times \Delta \to \Delta$ be the semi-flow defined by Equation
  \eqref{eq:semiflow}.
  \begin{propenum}
  \item \label{prop:FlowSmooth} For all $g \in \Delta$, $t \mapsto \phi(t,g)$ is
    $\mathcal{C}^\infty$ and its repeated derivatives are bounded.
  \item \label{prop:FlowLipschitz} For all $t \in \mathbb{R}_+$, $g \mapsto \phi(t, g)$ is Lipschitz
    with respect to $\norm{\cdot}$.
  \item \label{prop:FlowCont} For all $t \in \mathbb{R}_+$, $g \mapsto \phi(t, g)$ is continuous
    with respect to the pointwise convergence topology.
  \end{propenum}
\end{proposition}
\begin{remark}
  Stronger regularity property than \ref{prop:FlowLipschitz} could be proved as in finite
  dimension. Since we use only the Lipschitz continuity property, we didn't go further in this
  direction.
\end{remark}

\begin{proof}
  We begin with property \ref{prop:FlowSmooth}. The smoothness of the semi-flow with respect to the
  time variable can be shown by recurrence in a classical way. We have indeed:
  \[
    \partial_t \phi(t,g) = F(\phi(t,g)), \qquad \partial_{t}^2 \phi(t,g) = \mathcal{D} F[\phi(t,
    g)](\partial_t \phi(t,g)), \qquad \ldots
  \]
  Since $F$ is of class $\mathcal{C}^\infty$ and its repeated derivatives are bounded on $\Delta$
  (see \ref{prop:Fsmooth} and \ref{prop:FLipschitz} in Proposition \ref{prop:F}), the function
  $t \mapsto \phi(t,g)$ is of class $\mathcal{C}^\infty$ and its repeated derivatives are bounded
  for all $g\in \Delta$.
  
  \medskip

  We prove \ref{prop:FlowLipschitz}.  Recall that, since $\phi$ is the semi-flow associated to
  Equation \eqref{eq:SISODE}, the following equality holds for all $g \in \Delta$ and
  $t \in \mathbb{R}_+$:
  \begin{equation}\label{eq:semiflow_int}
    \phi(t, g) = g + \int_0^t F(\phi(s,g)) \, \mathrm{d}s.
  \end{equation}
  Let $g,h \in \Delta$. We have the following control:
  \begin{align*}
    \norm{\phi(t, g) - \phi(t,h)} &\leq \norm{g - h}
                                    + \int_0^t \norm{F(\phi(s,g)) - F(\phi(s,h))} \, \mathrm{d}s \\
                                  &\leq \norm{g - h}
                                    + C \int_0^t \norm{\phi(s,g) - \phi(s,h)} \, \mathrm{d}s,
  \end{align*}
  where $C$ is the Lipschitz coefficient of $F$ on $\Delta$ (see Proposition
  \ref{prop:F}~\ref{prop:FLipschitz}). We conclude by applying Gr\"{o}nwall's inequality.

  \medskip

  We prove property \ref{prop:FlowCont}.  Let $(g_n, \; n \in \mathbb{N})$ be a sequence of
  functions in $\Delta$ converging pointwise toward $g \in \Delta$. We define for
  $n \in \mathbb{N}$:
  \[
    \overline{g}_n = \sup\limits_{j \geq n} g_j 
    \quad\text{and}\quad
    \underline{g}_n = \inf\limits_{j \geq n} g_j.
  \]
  The sequence $(\overline{g}_n, \, n \in \mathbb{N})$ is non-increasing while
  $(\underline{g}_n, \, n \in \mathbb{N})$ is non-decreasing. We also have
  $\underline{g}_n \leq g_n \leq \overline{g}_n$ for all natural number $n$. Since the semi-flow is
  order-preserving by Proposition \ref{prop:Ordpres}, the following inequalities hold for all
  $(t, x) \in \mathbb{R}_+ \times \Omega$ and all $n \in \mathbb{N}^*$:
  \begin{equation}
    \label{eq:ineq}
    \phi(t,\underline{g}_{n-1})(x) \leq \phi(t,\underline{g}_{n})(x) \leq \phi(t, g_{n})(x)
    \leq \phi(t,\overline{g}_{n})(x) \leq \phi(t,\overline{g}_{n-1})(x).
  \end{equation}
  Thus, we can define two measurable functions $v,w \, \colon \mathbb{R}_+ \times \Omega \to [0,1]$
  by:
  \[
    v(t,x) = \lim\limits_{n \to \infty} \phi(t,\underline{g}_n)(x), \qquad w(t,x) =
    \lim\limits_{n\to\infty} \phi(t,\overline{g}_n)(x),
  \]
  for all $(t, x) \in \mathbb{R}_+ \times \Omega$. Notice that $v(t,x) \leq w(t,x)$ by construction.

  Fix $x \in \Omega$ and $t\geq 0$. We have:
  \[
    \phi(t,\overline{g}_n)(x) = \overline{g}_n(x) + \int_0^t F(\phi(s,\overline{g}_n))(x) \,
    \mathrm{d}s.
  \]
  The sequence of functions $(\overline{g}_n(x), \; n\in \mathbb{N})$ converges to $g(x)$ while the
  sequence of functions $(\phi(s,\overline{g}_n), \; n\in \mathbb{N})$ converges pointwise to
  $w(s,\cdot) \in \Delta$ for all $s \geq 0$. By continuity (see Proposition
  \ref{prop:F}~\ref{prop:Fproducttopo}), $F(\phi(s,\overline{g}_n))(x)$ converges to
  $F(w(s, \cdot))(x)$. Furthermore, the functions $s \mapsto F(\phi(s,\overline{g}_n))(x)$ are
  uniformly bounded since $F$ is bounded on $\Delta$ (see Proposition
  \ref{prop:F}~\ref{prop:FLipschitz}). Hence, we deduce from the dominated convergence theorem that:
  \[ 
    w(t,x) = g(x) + \int_0^t F(w(s,\cdot))(x) \, \mathrm{d}s.
  \]

  The previous equality is true for all $x \in \Omega$ and $t \geq 0$. Since $t \mapsto \phi(t,g)$
  is the only solution of \eqref{eq:SIS} having $g$ as initial condition, we have necessarily
  $w(t,\cdot) = \phi(t,g)$. We prove that $v(t, \cdot) = \phi(t,g)$ the same way. Letting $n$ go to
  infinity in \eqref{eq:ineq} proves that $\phi(t, g_n)$ converges pointwise to $\phi(t,g)$, for all
  $t \geq 0$.
\end{proof}

\subsection{Equilibria}
A function $g \in \Delta$ is an equilibrium of the dynamical system $(\Delta, \phi)$ (also called a
stationary point) if for all $t \in \mathbb{R}_+$, $\phi(t,g) = g$. The latter assertion is
equivalent to $F(g) =0$. The function equal to $0$ everywhere is a trivial stationary point. In
mathematical epidemiology, the other equilibria, if they exist, are called \textit{endemic states}
because they model a situation where the infection is constantly maintained at a baseline level in
the population.

The following result gives an easy way to identify those special states in the system. It is a
well-known fact in dynamical system theory which we prove anyway.
\begin{proposition}[Limit points are equilibria]\label{prop:pointwise_conv_equilibrium}
  Let $g \in \Delta$. If $t \mapsto \phi(t,g)$ converges pointwise to a limit $h^* \in \Delta$ when
  $t$ goes to $\infty$, then the function $h^*$ is an equilibrium.
\end{proposition}
\begin{proof}
  For all $x \in \Omega$ and $s\geq 0$, we have
  \[
    \phi(s, h^*)(x) = \lim\limits_{t \to \infty} \phi(s, \phi(t, g))(x) = \lim\limits_{t\to\infty}
    \phi(s + t, g)(x) = h^*(x),
  \]
  where the first inequality follows from the continuity of $\phi$ with respect to the pointwise
  convergence topology given in Proposition \ref{prop:Flow}~\ref{prop:FlowCont}. Thus, $h^*$ is an
  equilibrium.
\end{proof}

In the next remark, we check that any equilibrium is continuous with respect to an intrinsic
distance on $\Omega$ based on $\kappa$ and $\gamma$.
\begin{remark}[Continuity of the equlibria]
We consider for all $x,y \in \Omega$:
\[
  r(x,y) = \norm{\kappa(x, \cdot) - \kappa(y, \cdot)}_{\mathrm{TV}} + \abs{\gamma(x) - \gamma(y)},
\]
where $\norm{\cdot}_{\mathrm{TV}}$ is the total variation norm. The function $r$ defines a
pseudo-metric on the space $\Omega$. This pseudo-metric can be thought as an extension of the
neighborhood distance on graphons (see \cite[Section~13.3]{lovasz_large_2012}). Notice that the
Borel $\sigma$-field associated to the topology defined by $r$ is included in $\mathscr{F}$ since
$\gamma$ is measurable and $\kappa$ is a kernel.  \medskip

We have that if $h^*$ is an equilibrium of the dynamical system $(\Delta, \phi)$, then it is
continuous with respect to $r$.  Indeed, we have for all $x \in \Omega$:
\[
  h^*(x) = \frac{\lambda(x)}{\lambda(x) + \gamma(x)} \quad\text{with}\quad \lambda(x) = \int_\Omega
  h^*(z) \, \kappa(x, \mathrm{d}z).
\]
Both $\lambda$ and $\gamma$ are continuous with respect to $r$ and the function
$(a,b) \mapsto a/(a+b)$ is continuous on $\mathbb{R}_+ \times \mathbb{R}_+^*$. This implies that
$h^*$ is continuous.
\end{remark}

\subsection{The maximal equilibrium}
As a consequence of Proposition \ref{prop:deltainv} and Corollary \ref{cor:nodecreas}, the path
$t \mapsto \phi(t, 1)$ is non-increasing and bounded below by $0$. Thus, the path
$t \mapsto \phi(t, 1)$ converges pointwise to a limit say $g^*$ when $t$ goes to infinity:
\begin{equation}
  \label{eq:gstar}
  g^*(x) = \lim\limits_{t \to +\infty} \phi(t, 1)(x), \qquad \forall x \in \Omega.
\end{equation}

\begin{proposition} \label{prop:limitequilibrium} Let $g^*$ be defined by \eqref{eq:gstar}. We have
  the following properties.
  \begin{propenum}
   \item \label{prop:limitequilibriumI} The function $g^*$ is the maximal
     equilibrium of the dynamical system $(\Delta, \phi)$,
     \textit{i.e.}, if $h^*$ is an equilibrium, then $h^* \leq g^*$. 
   \item \label{prop:limitequilibriumII} For all $g^* \leq g \leq 1$, $\phi(t,g)$ converges
     pointwise to $g^*$ as $t$ goes to infinity.
  \end{propenum}
\end{proposition}

\begin{proof}
  We first prove property \ref{prop:limitequilibriumI}. The function $g^*$ is an equilibrium
  according to Proposition \ref{prop:pointwise_conv_equilibrium}. Let $h^*$ be another equilibrium
  in $\Delta$. By Proposition \ref{prop:Ordpres}, we have that $h^* = \phi(t, h^*) \leq \phi(t,1)$
  for all $t$. Letting $t$ goes to infinity, we obtain that $h^* \leq g^*$ and thus $g^*$ is the
  maximal equilibrium.

  \medskip  
  
  Now we prove property \ref{prop:limitequilibriumII}. Let $g^* \leq g \leq 1$. By Proposition
  \ref{prop:Ordpres}, we have:
  \[
    g^* \leq \phi(t, g) \leq \phi(t,1).
  \]
  Then, \eqref{eq:gstar} implies that $\phi(t,g)$ converges to $g^*$ as $t$ tends to infinity for
  the pointwise convergence.
\end{proof}
\begin{remark}\label{rem:g_star_inf_1}
  Since $\gamma(x) >0$ for all $x \in \Omega$ according to Assumption \ref{Assum0} and $F(g^*) = 0$,
  we have that $g^*(x) < 1$ for all $x \in \Omega$.
\end{remark}

In general, $g^*$ cannot be computed thanks to a closed-form expression even for the
finite-dimensional model. However, if the function $x \mapsto \kappa(x, \Omega)/\gamma(x)$ is
constant, then the formula used for the one-group model can be extended.

\begin{proposition}
\label{prop:g*=cst}
Suppose that there exists $C \in \mathbb{R}_+$ such that $\kappa(x, \Omega)/\gamma(x) = C$ for all
$x \in \Omega$. Then, $g^*$ is a constant function equal to $\max(0, 1-1/C)$.
\end{proposition}
\begin{proof}
  It is straightforward to check that the function $x \mapsto \max(0, 1-1/C)$ is an
  equilibrium. Now, we prove that it is maximal. Let $h^* \in \Delta$ be an equilibrium. From
  $F(h^*)/\gamma = 0$, we obtain the inequality:
  \[
    h^* \leq C (1 - h^*) \norm{h^*}.
  \]
  Taking a sequence $(x_n, \, n \in \mathbb{N})$ such that $h^*(x_n)$ converges to $\norm{h^*}$, we
  obtain at the limit that $\norm{h^*} \leq C (1 - \norm{h^*}) \norm{h^*}$. It follows that
  $\norm{h^*} \leq \max(0, 1-1/C)$.
\end{proof}

Since we cannot determine $g^*$ in the general case, the important question that naturally arises is
to find out whether the epidemic can survive in the population or if it will die out whatever the
initial condition is, \textit{i.e.}, we have to determine if $g^*(x)=0$ for all $x \in \Omega$. In
the following, we answer this question with Assumption \ref{Assum1} which imposes further conditions
on the transmission kernel $\kappa$ and the recovery rate $\gamma$.


\section{Kernels with density}\label{sec:dense}

In this section, we introduce some tools that we will use in Section \ref{sec:equilibrium}.

\subsection{Banach lattices}\label{subsec:Banachlattice}

For any notion not explained in the text, we refer the reader to the standard texts \cite{zaanen}
and \cite{schaefer_banach_1974} on the subject.  We recall that a partial order is a binary relation
$\leq$ over a set $X$ which is reflexive, antisymmetric and transitive. The inverse (or converse) of
$\leq$, denoted $\geq$, is the relation that satisfies $y \geq x$ if and only if $x \leq y$. A lower
bound of a subset $S$ of the partially ordered set $(X, \leq)$ is an element $a$ of $X$ such that
$a \leq x$ for all $x$ in $S$. A lower bound $a$ of $S$ is called an infimum (or greatest lower
bound) of $S$ if for all lower bounds $y$ of $S$ in $X$, $y \leq a$ ($a$ is larger than or equal to
any other lower bound). Similarly, an upper bound of a subset S of a partially ordered set
$(X, \leq)$ is an element $b$ of $X$ such that $b \geq x$ for all $x$ in $S$. An upper bound $b$ of
$S$ is called a supremum (or least upper bound) of $S$ if $b$ is less than any other upper
bound. Infima and suprema do not necessarily exist. However, if an infimum or supremum does exist,
it is unique.

A lattice is an ordered set such that a subset consisting of two points has a supremum and an
infimum. In a lattice $X$, the infimum and supremum of the subset $\set{x,y} \subset X$ are denoted
$x \wedge y$ and $x \vee y$ respectively. By induction it is immediately evident that every finite
subset of $X$ has a supremum and an infimum.

\medskip

A Riesz space is a vector space $X$ endowed with a lattice structure (denoted $\leq$) such that, for
any $x,y \in X$:
\begin{itemize}
\item[-] Translation invariance: if $x \leq y$ then $x + z \leq y + z$ for all $z\in X$.
\item[-] Positive homogeneity: if $x \leq y$, then $\lambda x \leq \lambda y$, for all scalar
  $\lambda \geq 0$.
\end{itemize}
We define the absolute value $\abs{x}$ of an element $x$ of a Riesz space by
$\abs{x} = x \vee (-x)$. We proceed with some further definitions.

\begin{definition}
  A Banach lattice $(X, \leq, \norm{\cdot})$ is a Riesz space $(X, \leq)$ equipped with a complete
  norm $\norm{\cdot}$ and such that, for any $x,y \in X$, we have:
  \begin{equation}
   \label{eq:abs}
   \abs{x} \leq \abs{y} \implies \norm{x} \leq \norm{y}.
 \end{equation}
\end{definition}
In the Banach lattice $X$, the positive cone:
\[
  X_+ = \set{ x \in E \, \colon \, x \geq 0}.
\]
is a proper cone, as it is a closed (see Theorem 15.1 (ii) in \cite{zaanen}) convex set such that
$\lambda X_+ \subset X_+$ for all $\lambda \in \mathbb{R}_+$, and $X_+ \cap (- X_+) = \set{0}$. It
is also a reproducing cone ($X=X_+-X_+$) as every element $x$ in $X$ can be decomposed as
$ x = (x\vee 0) - ((-x)\vee 0)$ and $y\vee 0\in X_+$ for all $y\in X$.

\subsection{Spectral analysis in Banach lattices}

In this section, we present some results of spectral analysis in Banach spaces. Let
$(X, \norm{\cdot})$ be a Banach space. We recall that the spectrum $\sigma(A)$ of a bounded operator
$A$ on $X$ is the of all complex numbers $\lambda$ such that $A - \lambda \mathrm{Id}$ does not have
a bounded inverse operator. It is well known that the spectrum of a bounded operator is a compact
set in $\mathbb{C}$.  The essential spectrum $\sigma_{\text{ess}}(A)\subset \sigma(A)$ is the set of
complex numbers $\lambda$ such that $A - \lambda \mathrm{Id}$ is not a Fredholm operator with index
$0$.

For a bounded operator $A$ on $X$, the spectral bound, the spectral radius and the essential
spectral radius are defined as:
\begin{align}
  \label{eq:spectral_bound}
  s(A) &= \sup \set{ \operatorname{Re}(\lambda) \, \colon \, \lambda \in
         \sigma(A)},  \\ 
  \label{eq:spectral_radius}
  r(A) &=  \sup \set{ \abs{\lambda} \, \colon \, \lambda \in
         \sigma(A)}=\lim_{n\rightarrow +\infty }
         \norm{A^n}^{1/n}=\inf_{n\in \mathbb{N}^* }
         \norm{A^n}^{1/n},  \\
  \label{eq:essential_spectral_radius}
  r_{\mathrm{ess}}(A) &=  \sup \set{ \abs{\lambda} \, \colon \, \lambda \in \sigma_{\text{ess}}(A)},
\end{align}
respectively, with the convention that $\sup \emptyset=0$. We refer to
\cite[Section~I.4]{edmunds_spectral_2018} for definitions of other essential spectra, which however
define the same essential spectral radius (in our setting $\sigma_{\text{ess}}(A)$ corresponds to
the essential spectrum $\sigma_{e4}(A) $ defined p.~37 in \cite{edmunds_spectral_2018}). As
$\sigma_{\text{ess}}(A)\subset \sigma(A)$, we get:
\begin{equation}\label{eq:radius_ineq}
  r_{\text{ess}}(A) \leq r(A) \leq \norm{A}.
\end{equation}

\medskip

The spectral theory of positive bounded operator on Banach lattice extends the Perron-Frobenius
theory in infinite dimension. Let $A$ be a positive operator on a Banach lattice
$(X, \leq, \norm{\cdot})$ such that its spectral radius $r(A)$ is positive. Recall $X^\star_+$ is
the dual cone of $X_+$.  A vector $x \in X_+ \backslash \set{0}$ (resp.
$x^\star\in X^\star_+ \backslash \set{0}$) such that $Ax = r(A)x$ (resp.
$A^\star x^\star = r(A) x^\star$) is called a right (resp. left) Perron eigenvector. We have the
following important result.

\begin{theorem}\label{th:spectralradius}
  Let $(X, \leq, \norm{\cdot})$ be a Banach lattice. Let $A, B$ be positive bounded operators on
  $X$. We have the following properties.
  \begin{propenum}
  \item \label{th:spectralradiusII} If $B - A$ is a positive operator, then $r(A) \leq r(B)$.
  \item \label{th:spectralradiusI} The spectral radius $r(A)$ belongs to $\sigma(A)$ and thus
    $r(A) = s(A)$.
  \item \label{th:KR} If $r_{\mathrm{ess}}(A) < r(A)$, then, there exists
    $x \in X_+ \backslash \set{0}$ such that: $Ax = r(A) x$.
  \end{propenum}
\end{theorem}
\begin{proof}
  Property \ref{th:spectralradiusII} is proved in \cite[Theorem~4.2]{marek}. Property
  \ref{th:spectralradiusI} is proved in \cite{Schaefer1960} (notice that \eqref{eq:abs} implies that
  $X_+$ is normal in the setting of \cite{Schaefer1960}), see also \cite[Lemma
  41.1.(ii)]{zaanen}. Property \ref{th:KR} was shown by Nussbaum in
  \cite[Corollary~2.2]{nussbaum_eigenvectors_1981} (notice that a reproducing cone is total), where
  the essential spectrum in \cite{nussbaum_eigenvectors_1981} is defined in
  \cite{nussbaum_radius_1970} and corresponds to $\sigma_{e5}(A)$ in
  \cite[p.~37]{edmunds_spectral_2018}. However, the essential spectral radius of $\sigma_{e5}(A)$ is
  equal to $r_{\mathrm{ess}}(A)$ the essential spectral radius of $\sigma_{e4}(A)$, according to
  \cite[Theorem I.4.10]{edmunds_spectral_2018}.
\end{proof}

It $A$ is assumed to be a compact operator, then Theorem \ref{th:spectralradius}~ \ref{th:KR} is the
so called Krein-Rutman theorem, see \cite[Theorem 41.2]{zaanen}.  We will also need the following
result proved in \cite[Propositions~2.1-2.2]{forster}.

\begin{proposition}[Collatz-Wielandt inequality]\label{prop:Collatz}
  Let $(X, \leq, \norm{\cdot})$ be a Banach lattice and $A$ be a positive bounded operator on
  $X$. We have:
  \[
    \sup \set{ \lambda \in \mathbb{R}\, \colon \, \exists x \in X_+ \backslash \set{0}, \, Ax \geq
      \lambda x } \leq r(A).
  \]
\end{proposition}

\subsection{The Banach lattice of bounded measurable functions}\label{subsec:Linfty}

The Banach space $(\mathscr{L}^\infty, \norm{\cdot})$ equipped with the partial order $\leq$ defined
by the proper cone $\mathscr{L}^\infty_+$ from \eqref{eq:coneK} is a Banach lattice.

Let $\nu$ be a finite signed measure on $(\Omega, \mathscr{F})$. For $g\in\mathscr{L}^\infty$, we
write $\braket{\nu,g}=\int_\Omega g(x) \, \nu(\mathrm{d}x)$ and thus identify $\nu$ as an element of
$\mathscr{L}^{\infty,\star}$, the dual space of $\mathscr{L}^{\infty}$. (Recall that
$\mathscr{L}^{\infty,\star}$ can be identified as the space of bounded and finitely additive signed
measures on $(\Omega, \mathscr{F})$.)

Let $\mu$ be a given finite positive measure on $(\Omega, \mathscr{F})$. For $q\in (1, +\infty)$,
denote by $(L^q(\mu), \norm{\cdot}_q)$ the usual Banach space of real-valued measurable functions
$f$ defined on $(\Omega,\mathscr{F})$ such that
$\norm{f}_q=\left(\int_\Omega |f(x)|^q \, \mu(\mathrm{d}x)\right)^{1/q}$ is finite and quotiented by
the equivalence relation of the $\mu$-almost everywhere equality.

Let $\iota$ be the natural linear application $\iota$ from $\mathscr{L}^\infty$ to $L^p(\mu)$, with
$p=q/(q-1)$ the conjugate of $q$, and $\iota^\star$ its dual. For $f\in L^q(\mu)$, we can see
$\iota^\star(f)$ as the bounded $\sigma$-finite signed measure $f(x) \mu(\mathrm{d}x)$ elements of
$\mathscr{L}^{\infty,\star}$.  By convention, for $f \in L^q(\mu)$ and $g$ in $\mathscr{L}^\infty$,
we write:
\[
  \braket{f,g}=\braket{\iota^\star(f),g}=\int_\Omega f(x) g(x) \, \mu(\mathrm{d}x).
\]

Let $\mathsf{k}$ be a non-negative measurable function defined on
$(\Omega\times \Omega, \mathscr{F}\otimes \mathscr{F})$ such that
$\sup_{x \in \Omega} \int \mathsf{k}(x,y) \, \mu(\mathrm{d}y) < \infty$.  We define the integral
operator $T_\mathsf{k}$ as the operator $T_\kappa$ defined by \eqref{eq:def-T} with kernel
$\kappa(x, \mathrm{d}y)= \mathsf{k}(x,y) \, \mu(\mathrm{d}y)$. Let $q\in (1, +\infty )$. We assume
the following condition holds:
\begin{align}
\label{eq:strongintegrability2}
  \sup\limits_{x \in \Omega} \int \mathsf{k}(x,y)^q \, \mu(\mathrm{d}y) < \infty. 
\end{align}
Then, we can also define the bounded operator:
\begin{align*}
  \tilde T_\mathsf{k} \quad \colon \quad L^p(\mu) &\to \mathscr{L}^\infty \\
  \nonumber
  g &\mapsto \left( x \mapsto \int_\Omega g(y) \, \mathsf{k}(x,y)  \,
      \mu(\mathrm{d}y) \right). 
\end{align*}
Using the density of the range of $\iota$, we get that
$T_{\mathsf{k}} = \tilde{T}_{\mathsf{k}} \iota$. We also define the bounded operator
$\hat T_{\mathsf{k}}$ from $L^p(\mu)$ to $L^p(\mu)$:
\begin{equation}
   \label{eq:def-hat-Tk}
   \hat T_{\mathsf{k}} = \iota \tilde{T}_{\mathsf{k}}.
\end{equation}
And we have the following commutative diagram:
\[
\begin{tikzcd}
  \mathscr{L}^\infty \arrow[dd, "T_{\mathsf{k}}"] \arrow[rr, "\iota"] & & L^p(\mu) \arrow[lldd,
  "\tilde{T}_{\mathsf{k}}"] \arrow[dd, "\hat
  T_{\mathsf{k}}"] \\
  &  &  \\
  \mathscr{L}^\infty \arrow[rr, "\iota"] & & L^p(\mu)
\end{tikzcd}
\]

The following lemma has a fundamental importance for the development of Section
\ref{sec:equilibrium}. The last property of the following Lemma on connected integral operator is
part of the Perron-Jentzsch theorem, see \cite[Theorem V.6.6 and Example
V.6.5.b]{schaefer_banach_1974}.

\begin{lemma}\label{lem:schae}
  Let $\mathsf{k}$ be a non-negative measurable function defined on
  $(\Omega\times \Omega, \mathscr{F}\otimes \mathscr{F})$ such \eqref{eq:strongintegrability2} holds
  for some $q\in (1, +\infty )$.  Then, the positive bounded operators
  $T_{\mathsf{k}} \, \colon \, \mathscr{L}^\infty \to \mathscr{L}^\infty$ and
  $T_{\mathsf{k}} \, \colon \, L^p(\mu) \to L^p(\mu)$, with $p=q/(q-1)$, satisfies:
  \begin{propenum}
  \item \label{lem:schae0} If $g=0$ $\mu$-a.e., then we have $T_{\mathsf{k}} g = 0$.

  \item \label{lem:schaeII} The operator $T_{\mathsf{k}}$ is weakly compact.
  \item \label{lem:schaeIII} The operators $T_{\mathsf{k}}^2$ and $\hat T_{\mathsf{k}}$ are compact.
  \item \label{lem:schaeIII-r} The operators $T_{\mathsf{k}}$ and $\hat T_{\mathsf{k}}$ have the
    same spectrum, and thus $r(T_{\mathsf{k}})=r(\hat T_{\mathsf{k}})$.

  \item \label{lem:schaeIV} If $r(T_{\mathsf{k}})>0$, then the operator $T_{\mathsf{k}}$ has a right
    Perron eigenvector in $\mathscr{L}^\infty_+ \backslash \set{0}$ and a left Perron eigenvector in
    $L^q_+(\mu) \backslash \set{0} \subset \mathscr{L}^{\infty, \star}_+ \backslash \set{0}$.
  \item \label{lem:schaeVI} If $\mathsf{k}$ is connected in the sense of Assumption
    \ref{Assum_connectivity}, then $r(T_{\mathsf{k}})>0$ and the right and left Perron eigenvector
    are unique (up to a multiplicative constant) and are $\mu$-a.e. positive, with the left Perron
    eigenvector seen as an element of $L^q_+(\mu) \backslash \set{0}$.
\end{propenum}
\end{lemma}

\begin{proof}
  Property \ref{lem:schae0} is straightforward.

  \medskip

  We prove property \ref{lem:schaeII}.  The natural linear application $\iota$ (resp.  the operator
  $\tilde{T}_{\mathsf{k}}$) from $\mathscr{L}^\infty$ to $L^p(\mu)$ (resp. from $L^p(\mu)$ to
  $\mathscr{L}^\infty$) is weakly compact thanks to \cite[Corollary IV.8.2 and Corollary
  VI.4.3]{df1}, which state that $L^p(\mu)$ is reflexive and that a bounded linear operator taking
  values in a reflexive Banach space is weakly-compact. In particular, we deduce that the operator
  $T_{\mathsf{k}}$, which is the product of a weakly compact operator and a bounded operator, is
  weakly compact, see \cite[Theorem VI.4.5]{df1}.  \medskip

  We prove property \ref{lem:schaeIII}. Notice that $\norm{f\vee g} = \norm{f}\vee \norm {g}$ for
  all $f, g \in \mathscr{L}^\infty_+$.  Thus the Banach lattice $\mathscr{L}^\infty$ is an AM-space,
  according to \cite[Definition II.7.1]{schaefer_banach_1974} and thus enjoys the Dunford-Pettis
  property thanks to \cite[Theorem II.9.9]{schaefer_banach_1974}. Then \cite[Corallary
  II.9.1]{schaefer_banach_1974} gives that the operator $T_{\mathsf{k}}^2$ is compact. And we also
  deduce from \cite[Theorem II.9.7]{schaefer_banach_1974} that $\hat T_{\mathsf{k}}$ is compact.
  This ends the proof of property \ref{lem:schaeIII}.  \medskip

  We prove property \ref{lem:schaeIII-r}.  If $\Omega$ is finite then the operators
  $ T_{\mathsf{k}}$ and $\hat T_{\mathsf{k}}$ coincide and there is nothing to prove.  So, we assume
  that $\Omega$ is infinite.  In this case, $\sigma_{\text{ess}}(T_{\mathsf{k}})$ and
  $\sigma_{\text{ess}}(\hat T_{\mathsf{k}})$ are non empty according to \cite[Footnote 2,
  p.~243]{kato2013perturbation}.  As $ \hat T_{\mathsf{k}}$ and $ T_{\mathsf{k}}^2$ are compact, we
  deduce from \cite[Theorems VII.4.5 and VII.4.6]{df1} respectively, that the essential spectra of
  $ \hat T_{\mathsf{k}}$ and $ T_{\mathsf{k}}$ are reduced to $\{0\}$, and that the non-null
  elements of their spectrum are eigenvalues. Then, use that
  $\iota T_{\mathsf{k}}= \hat T_{\mathsf{k}}\iota$ and property \ref{lem:schae0}, to deduce that if
  $f\in\mathscr{L}^\infty\setminus\set{0}$ is an eigenvector of $ T_{\mathsf{k}}$, then $\iota(f)$
  belongs to $L^p(\mu)\backslash\set{0}$ thanks to property \ref{lem:schae0} and that $\iota(f)$ is
  thus an eigenvector of $\hat T_{\mathsf{k}}$ corresponding to the same eigenvalue. If
  $v\in L^p(\mu)\backslash\set{0}$ is an eigenvector of $ \hat T_{\mathsf{k}}$ corresponding to the
  eigenvalue $\lambda$, then $f=\tilde T_{\mathsf{k}}(v)$ belongs to $\mathscr{L}^\infty$ and
  $f\neq 0$ (as $\iota (f)=\hat T_{\mathsf{k}}(v)=\lambda v$). We have
  $T_{\mathsf{k}}(f)=\tilde T_{\mathsf{k}} \iota \tilde T_{\mathsf{k}}(v)= \tilde
  T_{\mathsf{k}}(T_{\mathsf{k}}(v))=\lambda \tilde T_{\mathsf{k}}(v) =\lambda f$. Thus $\lambda$ is
  also an eignevalue of $T_{\mathsf{k}}$. We deduce that
  $\sigma(T_{\mathsf{k}})=\sigma(\hat T_{\mathsf{k}})$.

  \medskip

  We prove property \ref{lem:schaeIV}. We have seen that
  $\sigma_{\text{ess}}(T_{\mathsf{k}})\subset\set{0}$ and thus $r_{\mathrm{ess}}(T_{\mathsf{k}})=0$.
  According to Theorem \ref{th:spectralradius}~\ref{th:KR} (or the Krein-Rutman theorem) there
  exists a right Perron eigenvector for $T_{\mathsf{k}}$.  Since $\hat T^\star_{\mathsf{k}}$ is a
  compact operator, thanks to Schauder Theorem \cite[Theorem VI.5.2]{df1}, with the same spectrum as
  $\hat T_{\mathsf{k}}$, thanks to \cite[Lemma VII.3.7]{df1}, and which is clearly positive, we
  deduce from Theorem \ref{th:spectralradius}~\ref{th:KR} that there exists a right Perron
  eigenvector, $v^\star\in L^q_+(\mu)\backslash\set{0}$, for $\hat T_{\mathsf{k}}^\star$.  Since
  $T^\star_{\mathsf{k}} \iota^\star= \iota^\star \hat T^\star_{\mathsf{k}}$, we deduce that
  $\iota^\star (v^\star)$, and thus $v^\star$ by convention, is also a left Perron eigenvector for
  $ T_{\mathsf{k}}$. This gives property \ref{lem:schaeIV}.

  \medskip

  Now, let us prove property \ref{lem:schaeVI}. Set
  $\lambda=r( T_{\mathsf{k}})=r(\hat T_{\mathsf{k}})$, see property \ref{lem:schaeIII-r}.  According
  to the Perron-Jentzsch theorem \cite[Theorem V.6.6 and Example V.6.5.b]{schaefer_banach_1974},
  since $\mathsf{k}$ is connected in the sense of Assumption \ref{Assum_connectivity}, we have
  $\lambda>0$ and there exists a unique (up to a multiplicative constant) eigenvector $v$ of
  $\hat T_{\mathsf{k}}$ associated to the eigenvalue $\lambda$, and it can chosen such that
  $\mu$-a.e.  $v>0$.  According to the proof of property \ref{lem:schaeIII-r}, we get that
  $f=\tilde T_{\mathsf{k}}(v)$ is an eigenvector of $T_{\mathsf{k}}$ associated to $\lambda$. Notice
  that $f\geq 0$ as $\mathsf{k}\geq 0$.  Since $\iota(f)=\hat T_{\mathsf{k}}(v)=\lambda v$, we
  deduce that $\mu$-a.e.  $f>0$. Assume that $g\in\mathscr{L}^\infty\setminus\set{0}$ is a right
  Perron eigenvector of $T_{\mathsf{k}}$, then $\iota(g)$ is a right Perron eigenvector of
  $\hat T_{\mathsf{k}}$ and thus (up to a multiplicative constant chosen to be equal to $\lambda$),
  we have $\mu$-a.e.  $\iota(g)=\lambda v=\iota(f)$.  We deduce that $\mu$-a.e. $g-f=0$ and thanks
  to property \ref{lem:schae0}, we deduce that $\lambda(f-g)=T_{\mathsf{k}}(f-g)=0$.  So the right
  Perron eigenvector of $T_{\mathsf{k}}$ is unique and $\mu$-a.e. positive.

  Let $f^\star$ be a left Perron eigenvector of $T_{\mathsf{k}}$. Then
  $v^\star= \tilde T_{\mathsf{k}}^\star(f^\star)$ is an eigenvector of $\hat T_{\mathsf{k}}^\star$
  associated to $\lambda$ and $v^\star\in L^q_+(\mu)\backslash\set{0}$ as $\mathsf{k}\geq 0$. By the
  Perron-Jentzsch theorem, we get that $v^\star$ is unique (up to a multiplicative constant) and
  that $\mu$-a.e $v^\star>0$.  Since $\iota^\star (v^\star)=\lambda f^\star$, we deduce that
  $\mu$-a.e $f^\star>0$ and that $f^\star$ is unique (up to a multiplicative constant).
\end{proof}

\begin{remark}\label{rem:equilibrium_not_almost_zero}
  As a consequence of Lemma \ref{lem:schae}~\ref{lem:schae0}, under Assumption \ref{Assum1}, if
  $h^*$ is an equilibria which is $\mu$-a.e. equal to 0, then it is equal to 0 everywhere.
\end{remark}


\section{Infinite-dimensional SIS model when the kernel has a density}\label{sec:equilibrium}

The objective of this section is to study the long time behavior of the solutions of \eqref{eq:SIS}
under Assumption \ref{Assum0} and Assumption \ref{Assum1} (but for Section \ref{subsec:subcritical}
where the latter is not assumed).  Recall the definition of the spectral bound given in
\eqref{eq:spectral_bound}.  We will consider the spectral bound $s(T_k - \gamma)$ of the bounded
operator $T_k - \gamma$ on $\mathscr{L}^\infty$ to characterize three different regimes:
sub-critical, critical and super-critical, corresponding to the cases $s(T_k - \gamma)<,=, > 0$
respectively.  In the first part of the section, we establish a link between $s(T_k - \gamma)$ and
the basic reproduction number $R_0=r(T_{k/\gamma})$ associated to \eqref{eq:SIS}.

\subsection{Basic reproduction number and spectral bound}\label{subsec:R0}
Recall that Assumption \ref{Assum0} is in force.  If we assume $\inf \gamma > 0$, then the operator
$T_{\kappa/\gamma}$, where $\kappa/\gamma$ is the kernel defined by
$(\kappa/\gamma) (x, \mathrm{d}y)= \kappa (x, \mathrm{d}y)/\gamma(y)$ is bounded. Thieme's result
\cite[Theorem~3.5]{thieme_spectral_2009} implies the following proposition.

\begin{proposition}\label{prop:R0_Thieme}
  If $\inf \gamma > 0$, then $r(T_{\kappa/\gamma}) - 1$ has the same sign as $s(T_\kappa - \gamma)$
  (\textit{i.e.} these two numbers are simultaneously negative, zero, or positive).
\end{proposition}
\begin{proof}
  Consider the operators $A=T_\kappa-\gamma$ and $B=-\gamma$, where $-\gamma$ is the operator
  corresponding to the multiplication by $-\gamma$.  It is clear from
  \cite[Definition~3.1]{thieme_spectral_2009} that the operator $B$ is a resolvent-positive
  operator, as the operator $\lambda-B=\lambda + \gamma$ is invertible and its inverse is positive
  for $\lambda>0$. We also get that $s(B)=s(-\gamma)= -\inf \gamma<0$.  Let
  $Q = A + \norm{\gamma} $. The operator $Q$ is positive. And for $\lambda > r(Q)$, we get that
  $(\lambda - Q)^{-1}$ is also positive since, thanks to the Neumann series expansion, we have:
  \[
    (\lambda - Q)^{-1} = \sum\limits_{i=0}^\infty \frac{1}{\lambda^{i+1}} Q^i \geq 0.
  \]
  We deduce that $(\lambda -A)$ is invertible and its inverse is positive for
  $\lambda>r(Q) - \norm{\gamma}$.  Hence, $A$ is resolvent-positive.  Applying
  \cite[Theorem~3.5]{thieme_spectral_2009} (notice it is required that $\mathscr{L}^\infty_+$ is
  normal, which is the case, see \cite[Proposition 19.1]{deimling2010nonlinear}, as the norm
  $\norm{\cdot}$ is monotonic: $0\leq f\leq g$ implies $\norm{f}\leq \norm{g}$), we deduce that
  $s(A)$ has the same sign as $r(-(A-B) B^{-1}) -1= r(T_{\kappa/\gamma}) -1$.
\end{proof}

Notice that under Assumption \ref{Assum1}, we have by definition
$R_0=r(T_{\kappa/\gamma})=r(T_{k/\gamma})$, as  $k$
is the density of $\kappa$ with respect to $\mu$. In what follows, we also write 
 $T_k$ for $T_\kappa$.  
When $\gamma$ is not bounded away  from 0, but still positive, $R_0-1$ and
$s(T_k - \gamma)$ may not have the same sign. For instance, if one takes
$T_k   =  0$   and   $\inf   \gamma  =   0$,   then   we  clearly   have
$s(T_k  - \gamma)  = s(-\gamma)  =  0$ and  $R_0  = 0$.  
\medskip

As we only assume that  $\inf \gamma\geq 0$, we  will get  a  weaker result  but  this will  be  sufficient for  our
purpose.  
According to Assumption \ref{Assum1} (see
\eqref{eq:strongintegrability}), the operator $T_{k/\gamma} $ defined by
\eqref{eq:def-tk/g} 
 is a bounded operator on
$\mathscr{L}^\infty$ which satisfies the integrability condition of
Lemma \ref{lem:schae} with $\mathsf{k}(x,y) = k(x,y)/\gamma(y)$. 

\begin{proposition}\label{prop:R0_s}
Suppose Assumption \ref{Assum1} is in force. Then, the following assertions are equivalent:
\begin{propenum}
\item \label{prop:R0_sI} $s(T_k - \gamma) > 0$.
\item \label{prop:R0_sII} $R_0 > 1$.
\item \label{prop:R0_sIII} There exists $\lambda>0$ and $w \in \mathscr{L}^\infty_+ \backslash \set{0}$ such that:
\begin{equation}\label{eq:eigenfunction}
T_k(w) - \gamma w = \lambda w.
\end{equation} 
\end{propenum}
\end{proposition}
\begin{proof}
  It is immediate that property \ref{prop:R0_sIII} implies property \ref{prop:R0_sI}.

  \medskip
 
  We suppose property \ref{prop:R0_sI} and we prove \ref{prop:R0_sII}. Let
  $a\in (0,s(T_k - \gamma))$, so that $s(T_k - (\gamma+a)) = s(T_k - \gamma) - a >0$. Using
  Proposition \ref{prop:R0_Thieme} (with $\gamma$ replaces by $\gamma+a$), we get that
  $r\left(T_{k/(\gamma + a)}\right)>1$.  Since
  $r(T_{k/\gamma}) \geq r\left(T_{k/(\gamma + a)}\right)$ according to Theorem
  \ref{th:spectralradius}~\ref{th:spectralradiusII}, property \ref{prop:R0_sII} is shown.

  \medskip
 
  Now we assume property \ref{prop:R0_sII}, and we prove property \ref{prop:R0_sIII}. Consider for
  any non-negative real number $a\geq 0$, the function:
  \[
    \psi(a) = r(T_{k/(\gamma + a)}).
  \]
  Property \ref{prop:R0_sII} exactly means that:
  \begin{equation}
    \label{eq:psi_0}
    \psi(0) > 1.
  \end{equation}
  Moreover, it follows from the inequality
  $r(T_{k/(\gamma + a)}) \leq \norm{T_{k/(\gamma + a)}} \leq \norm{T_k} / a$ (use
  \eqref{eq:radius_ineq} for the first inequality), that:
  \begin{equation}\label{eq:psi_infinite}
    \lim\limits_{a \to \infty} \psi(a) = 0.
  \end{equation}
  Equation \eqref{eq:strongintegrability} of Assumption \ref{Assum1} enables to apply Lemma
  \ref{lem:schae}~\ref{lem:schaeIII} and we obtain that all the operators $T_{k/(\gamma + a)}$, for
  $a \in \mathbb{R}_+$, are power compact (as $T^2_{k/(\gamma + a)}$ is compact). According to
  \cite[Theorem p.~21]{konig1986}, their spectra are totally disconnected. Moreover, since the
  function $a \mapsto T_{k/(\gamma +a)}$ mapping $\mathbb{R}_+$ to $\mathcal{L}(\mathscr{L}^\infty)$
  is continuous, we also get thanks to \cite[Theorem 11]{newburgh1951} that the application
  $a \mapsto \sigma(T_{k/(\gamma + a)})$ mapping $\mathbb{R}_+$ to the set $\mathcal{K}(\mathbb{C})$
  of non-empty compact subsets endowed with the Hausdorff distance (see Section \ref{sec:Hausdorff}
  for the definition of the Hausdorff distance) is continuous. Hence, the function $\psi$ is
  continuous according to Lemma \ref{lem:Hausdorff_radius}. From the continuity of $\psi$ and
  Equations \eqref{eq:psi_0} and \eqref{eq:psi_infinite}, we conclude that there exists $\lambda>0$
  such that $\psi(\lambda) = 1$. According to Lemma \ref{lem:schae}~\ref{lem:schaeIV}, there exists
  a function $v \in \mathscr{L}^\infty_+\backslash \set{0}$ such that:
  \[
    T_k\left(\frac{v}{\gamma + \lambda}\right) = v.
  \]
  Then, Equation \eqref{eq:eigenfunction} holds with $w = v/(\gamma+\lambda)$ which proves property
  \ref{prop:R0_sIII}.
\end{proof}

\begin{remark}\label{rem:eigenvector_non_almost_zero}
  Using Lemma \ref{lem:schae}~\ref{lem:schae0}, it is easy to show that $w$ in Proposition
  \ref{prop:R0_s}~\ref{prop:R0_sIII} should satisfy $\int_\Omega w(x) \, \mu(\mathrm{d}x) >0$.
\end{remark}

The next result is stronger than the implication (i) $\implies$ (ii) in Proposition \ref{prop:R0_s}.
\begin{lemma}
  Under Assumption \ref{Assum1}, the following inequality holds:
  \begin{equation}
    \label{eq:R0_s_inequality}
    s(T_k - \gamma) \leq \max(\norm{\gamma} (R_0 - 1), 0).
  \end{equation}
\end{lemma}
\begin{proof}
  If $s(T_k - \gamma) \leq 0$, the result is obviously true. Suppose $s(T_k - \gamma) >0$.  Since
  $T_k - \gamma + \norm{\gamma}$ is a positive operator, Theorem
  \ref{th:spectralradius}~\ref{th:spectralradiusI} implies that
  $r(T_k - \gamma + \norm{\gamma})=s(T_k - \gamma + \norm{\gamma})$. Since
  $s(T_k - \gamma + \norm{\gamma})=s(T_k - \gamma) + \norm{\gamma}>\norm{\gamma} $, we obtain that:
  \[
    r(T_k - \gamma + \norm{\gamma}) >  \norm{\gamma}.
  \]
  Besides, we have $\norm{\gamma}\geq r(\norm{\gamma}- \gamma)$ according to Theorem
  \ref{th:spectralradius}~\ref{th:spectralradiusII} and
  $r(\norm{\gamma} - \gamma)\geq r_{\mathrm{ess}}(\norm{\gamma} - \gamma)$ according to Equation
  \eqref{eq:radius_ineq}. We deduce that:
  \[
    r(T_k - \gamma + \norm{\gamma} > r_{\mathrm{ess}}(\norm{\gamma} - \gamma).
  \]
  The operator $T_k$ is weakly compact thanks to Lemma \ref{lem:schae}~\ref{lem:schaeII} since $k$
  satisfies \eqref{eq:strongintegrability2}, see Assumption \ref{Assum1} and more precisely
  \eqref{eq:weakintegrability}.  Since $\mathscr{L}^\infty$ has the Dunford-Pettis property, see
  \cite[Section~II.9]{schaefer_banach_1974}, we deduce from \cite[Theorem~3.1]{latrach} (where
  $\sigma_{\text{ess}}(A)$ in our setting corresponds to $\sigma_{e5}(A)$ in \cite{latrach}) that
  $r_{\mathrm{ess}}(\norm{\gamma} - \gamma)=r_{\mathrm{ess}}(T_k-\gamma+ \norm{\gamma})$.
  Therefore, we get that:
  \[
    r(T_k - \gamma + \norm{\gamma}) >  r_{\mathrm{ess}}(T_k  - \gamma +
    \norm{\gamma}). 
  \]
  Hence, we can apply Theorem \ref{th:spectralradius}~\ref{th:KR} with the positive operator
  $T_k - \gamma + \norm{\gamma}$, to get the existence of a function
  $w \in \mathscr{L}^\infty_+ \backslash \set{0}$ such that:
  \begin{equation}
    T_k(w) - \gamma w = s(T_k - \gamma) w,
  \end{equation}
  where we used
  $r(T_k - \gamma + \norm{\gamma})=s(T_k - \gamma + \norm{\gamma}) =s(T_k - \gamma) + \norm{\gamma}$
  for the equality.  We have shown that one can actually take $\lambda = s(T_k - \gamma)$ in
  Equation \eqref{eq:eigenfunction}. Thus, we obtain:
  \[
    T_{k/\gamma} (\gamma w) = T_k(w) = (\gamma + s(T_k - \gamma)) w \geq \left(1 + \frac{s(T_k -
        \gamma)}{\norm{\gamma}}\right) \gamma w.
  \]
  According to Proposition \ref{prop:Collatz}, we conclude that:
  \begin{equation}\label{eq:R0minoration}
    R_0=r(T_{k/\gamma}) \geq 1 + \frac{s(T_k - \gamma)}{\norm{\gamma}}\cdot 
  \end{equation}
  We deduce that Equation \eqref{eq:R0_s_inequality} holds.
\end{proof}

We continue the study with a proposition about the perturbation of the operator $T_{k/\gamma}$. For
$g\in \Delta$, we define $R_0(g) = r(g T_{k / \gamma})$.
\begin{proposition}
  \label{prop:spectrumcontinuous}
  Suppose Assumption \ref{Assum1} holds.  The function $g \mapsto R_0(g)$ defined on $\Delta$ is
  non-decreasing and continuous with respect to the $L^1(\mu)$ topology.
\end{proposition} 
\begin{proof}
  The fact that $g \mapsto R_0(g)$ is non-decreasing is a direct consequence of
  Theorem~\ref{th:spectralradius}~\ref{th:spectralradiusII}.

  For for $g \in \Delta$, the bounded operator $A_g =\hat T_{\mathsf{k}}$ on $L^p(\mu)$ defined in
  Equation \eqref{eq:def-hat-Tk} with the kernel $\mathsf{k}(x,y)=g(x) k(x,y)/\gamma(y)$ is compact
  according to Lemma \ref{lem:schae}~\ref{lem:schaeIII}. According to Lemma
  \ref{lem:schae}~\ref{lem:schaeIII-r}, we have that for all $g\in \Delta$:
  \begin{equation}
    \label{eq:spectral_equality}
    R_0(g) = r(A_g).
  \end{equation}
  Besides, the function $g \mapsto A_g$ mapping $\Delta$ to $\mathcal{L}(L^p(\mu))$ is continuous
  with respect to the $L^p(\mu)$ norm.  We deduce from \cite[Theorem 11]{newburgh1951}, that the
  function $ g \mapsto \sigma(A_g)$ from $(\Delta, \norm{\cdot}_p)$ to
  $(\mathcal{K}(\mathbb{C}),d_{\mathrm{H}})$ is continuous, where $\mathcal{K}(\mathbb{C})$ is the
  set of non-empty compact subsets and $d_{\mathrm{H}}$ is the Hausdorff distance (see Section
  \ref{sec:Hausdorff} for the definition of the Hausdorff distance).  Using Lemma
  \ref{lem:Hausdorff_radius} and then Equation~\eqref{eq:spectral_equality}, we get that
  $g \mapsto R_0(g)$ defined on $(\Delta, \norm{\cdot}_p)$ is continuous.  In order to conclude, we
  notice that the topologies induced by $L^p(\mu)$ and $L^1(\mu)$ are equal on $\Delta$ because
  $\Delta$ is a bounded subset of $L^\infty(\mu)$.  This proves that $g \mapsto R_0(g)$ defined on
  $(\Delta, \norm{\cdot}_1)$ is continuous.
\end{proof}

\subsection{The subcritical regime: $s(T_\kappa - \gamma) < 0$}\label{subsec:subcritical}
Recall that Assumption \ref{Assum0} is in force.  Here, we show that in the subcritical regime, the
solutions of Equation \eqref{eq:SISODE} converge exponentially fast to $0$ in norm.

\begin{theorem}[Uniform exponential extinction]\label{th:sub}
  Suppose that $s(T_\kappa - \gamma) < 0$. Then, for all $c \in (0, -s(T_\kappa - \gamma))$, there
  exists a finite constant $\theta = \theta(c)$ such that, for all $g \in \Delta$, we have:
  \begin{equation}\label{eq:sub_conv}
    \norm{\phi(t,g)} \leq  \theta \norm{g}\ \mathrm{e}^{-ct}.
  \end{equation}
  In particular, the maximal equilibrium $g^*$ is equal to $0$ everywhere.
\end{theorem}
\begin{proof}
  Recall $T_\kappa-\gamma$ is a bounded operator. For all $t \in \mathbb{R}_+$, define:
\begin{equation}\label{eq:sub_conv_proof1}
  v(t) = \mathrm{e}^{t(T_\kappa - \gamma)} 1 = \sum\limits_{n \in
    \mathbb{N}} \frac{t^n}{n!}(T_\kappa - \gamma)^n \, 1. 
\end{equation}
We also have:
\[
  \mathrm{e}^{\norm{\gamma} t} v(t) = \mathrm{e}^{t(T_\kappa - \gamma+ \norm{\gamma})} 1 =
  \sum\limits_{n \in \mathbb{N}} \frac{t^n}{n!}(T_\kappa - \gamma+\norm{\gamma} )^n \, 1.
\]
As $T_\kappa - \gamma+ \norm{\gamma}$ is positive, we deduce that $v(t)\geq 0$.  As $T_\kappa$ is
positive, we deduce that:
\[
  v'(t) - F(v(t)) = (T_\kappa - \gamma)(v(t)) - F(v(t)) = v(t) T_\kappa(v(t)) \geq 0.
\]
Thus, the following inequality holds for all $g \in \Delta$ and all $t \geq 0$:
\[
  0 = \partial_t \phi(t,g) - F(\phi(t,g)) \leq v'(t) - F(v(t)).
\]
As $F$ is cooperative on $\Delta\times \mathscr{L}^\infty_+$, see Proposition
\ref{prop:F}~\ref{prop:Fcoop}, we can apply Corollary \ref{cor:Comparison} with
$K=\mathscr{L}^\infty_+$, $D_1 = \Delta$, $D_2 = \mathscr{L}^\infty_+$ and $u(t) = \phi(t,g)$ to
obtain that:
\begin{equation}\label{eq:sub_conv_proof2}
  \phi(t, g) \leq v(t) \quad \text{for all $t \in \mathbb{R}_+$.}
\end{equation}
Besides, since $T_k - \gamma$ is a bounded operator, its growth bound (\textit{i.e.}, the left
member of the equality below) is equal to its spectral bound according to
\cite[Theorem~I.4.1]{daleckij_stability_2005}:
\begin{equation}\label{eq:sub_conv_proof3}
  \inf \left\lbrace \eta \in \mathbb{R} \, \colon \, \sup\limits_{t \in \mathbb{R}_+} \mathrm{e}^{-
      \eta t} \norm{\exp(t(T_\kappa - \gamma))}
    < \infty  \right\rbrace = s(T_\kappa - \gamma).
\end{equation}
We deduce from Equations \eqref{eq:sub_conv_proof1}, \eqref{eq:sub_conv_proof2} and
\eqref{eq:sub_conv_proof3}, that for all $c \in (0, -s(T_\kappa - \gamma))$, there exists a finite
constant $\theta$ such that Equation \eqref{eq:sub_conv} is true.  In particular,
$t \mapsto \phi(t,1)$ converges uniformly to $0$. It then follows from Equation \eqref{eq:gstar}
that $g^*$ is equal to $0$ everywhere.
\end{proof}

\subsection{Critical regime:  $s(T_k -\gamma)=0$}
\label{subsec:critical}
Assumption \ref{Assum1} holds. Recall that $k$ is the density of the kernel $\kappa$ and that we
then write $T_k$ for $T_\kappa$. We give the main result of this section.

\begin{theorem}[Extinction at criticality]\label{th:critical}
  Suppose Assumption \ref{Assum1} is in force and $s(T_k - \gamma)=0$. Then the maximal equilibrium
  $g^*$ is equal to $0$ everywhere. In other words, for all $g \in \Delta$ and all $x \in \Omega$,
  we have that:
  \[
    \lim\limits_{t \to \infty} \phi(t, g)(x) = 0.
  \]
\end{theorem}
\begin{proof}
  Suppose, to derive a contradiction, that $g^*$ is not equal to $0$ $\mu$-almost everywhere. We
  know according to Remark \ref{rem:g_star_inf_1} that $1 - g^*$ is positive everywhere. Hence, we
  get:
  \begin{equation}
    \label{eq:R0_big}
    T_{k/\gamma}(\gamma g^*) =T_k(g^*)= \left( 1 + \frac{g^*}{1-g^*}\right) \gamma g^* \geq \gamma g^*.
  \end{equation}
  According to Proposition \ref{prop:Collatz}, $R_0$ is then greater than or equal to $1$.

  \medskip

  We will prove that in fact $R_0>1$.  Consider the set
  $A = \set{ x \in \Omega \, \colon \, g^*(x)>0}$. Equation \eqref{eq:R0_big} remains true by
  replacing $k$ by $k' = \mathds{1}_A \, k \, \mathds{1}_A$ (\textit{i.e.}
  $k'(x,y)=\mathds{1}_A(x)k(x,y) \mathds{1}_A(y)$):
  \begin{equation}
    \label{eq:R0_big_bis}
    T_{k'/\gamma}(\gamma g^*) = \left( 1 + \frac{g^*}{1-g^*}\right) \gamma g^* \geq \gamma g^*.
  \end{equation}
  Using Proposition \ref{prop:Collatz}, we get that $r(T_{k'/\gamma}) \geq 1$. Since Assumption
  \ref{Assum1} is in force, $T_{k'/\gamma}$ has a left Perron eigenvector $h$ in
  $L^q_+(\mu) \backslash \set{0}$ (see Lemma \ref{lem:schae}~\ref{lem:schaeIV}).  By multiplying
  both members of Equation \eqref{eq:R0_big_bis} by $h$ and integrating with respect to $\mu$, we
  obtain:
  \begin{equation}\label{eq:brakets}
    (r(T_{k'/\gamma}) - 1) \braket{h, \gamma g^*} = \braket{h, (g^*)^2 \gamma / (1 - g^*)}
  \end{equation}
  It is clear that $h \mathds{1}_{A^c} = 0$. Since $h \in L^q_+(\mu) \backslash \set{0}$, we have
  necessarily:
  \[
    \int_A h(x) \, \mu(\mathrm{d}x) > 0.
  \]
  Hence, both brakets in Equation \eqref{eq:brakets} are positive. Thus, we get that
  $r(T_{k'/\gamma})>1$.  Using Theorem \ref{th:spectralradius}~\ref{th:spectralradiusII} and that
  the operator $T_{k/\gamma} - T_{k'/\gamma}$ is positive, we deduce that
  $R_0\geq r(T_{k'/\gamma})>1$. This is in contradiction with Proposition \ref{prop:R0_s} which
  asserts that $R_0\leq 1$ as $s(T_k-\gamma)=0$.  Thus, we obtain that $\mu$-a.e. $g^*=0$.  We
  conclude using Remark~\ref{rem:equilibrium_not_almost_zero}.
\end{proof}

\subsection{Supercritical regime: $s(T_k - \gamma) > 0$}\label{subsec:supercritical}
Assumption \ref{Assum1} is in force in this section. We consider the case $s(T_k - \gamma) > 0$. We
will begin the analysis by proving that $g^*$ is different from $0$. Then, we will show the
convergence of the system to $g^*$.  \medskip

According to Proposition \ref{prop:spectrumcontinuous}, there exists $\varepsilon_0 \in (0,1)$ such
that, for all $\varepsilon \in (0, \varepsilon_0)$, $R_0((1 - \varepsilon)) >1$. For each
$\varepsilon\in (0, \varepsilon_0)$, Proposition \ref{prop:R0_s} ensures the existence of a vector
$w_\varepsilon \in \mathscr{L}^\infty_+ \backslash \set{0}$ and a positive real number
$\lambda(\varepsilon)>0$, such that:
  \begin{equation}\label{eq:eigenvector}
    (1 - \varepsilon) T_k(w_\varepsilon) = (\gamma + \lambda(\varepsilon)) w_\varepsilon.
  \end{equation}
  We can take $w_\varepsilon$ such that $\norm{w_\varepsilon} < \varepsilon$. Moreover, according to
  Remark \ref{rem:eigenvector_non_almost_zero}:
\begin{equation}
\int_\Omega w_\varepsilon(x) \, \mathrm{d}x > 0.
\end{equation}
Then, we get the following proposition.
\begin{proposition}[Increasing trajectory]\label{prop:decolle}
  Suppose Assumption \ref{Assum1} is in force and that $s(T_k - \gamma) > 0$.  For all
  $\varepsilon\in (0,\varepsilon_0)$, the map $t \mapsto \phi(t, w_\varepsilon)$ is non-decreasing.
\end{proposition}
\begin{proof}
  Let $\varepsilon\in (0,\varepsilon_0)$. We have:
  \[
    0 \leq \lambda(\varepsilon) w_\varepsilon = (1 - \varepsilon) T_k(w_\varepsilon) - \gamma
    w_\varepsilon \leq (1 - w_\varepsilon) T_k(w_\varepsilon) - \gamma w_\varepsilon =
    F(w_\varepsilon),
  \]
  where the last inequality holds because $\norm{w_\varepsilon} < \varepsilon$. We conclude using
  Proposition \ref{prop:propC}.
\end{proof}

Proposition \ref{prop:decolle} shows that the equilibrium $0$ is not asymptotically stable in the
following sense: we can find initial conditions arbitrarily close to $0$ in norm such that
$\phi(t,g)$ does not converge to $0$ pointwise. An easy consequence of Proposition
\ref{prop:decolle} is that $g^*$ is not $\mu$-a.e. equal to 0.

\begin{corollary}\label{cor:existenceendemic}
  If Assumption \ref{Assum1} is in force and $s(T_k - \gamma) > 0$, then we have:
  \[
  \int_\Omega g^*(x) \, \mu(\mathrm{d}x) > 0.
  \] 
\end{corollary}

We deduce from Proposition \ref{prop:decolle} that $t \mapsto \phi(t,w_\varepsilon)$ converges
pointwise as $t$ tends to infinity since $\phi(t,w_\varepsilon)\leq 1$ for all $t$. According to
Proposition \ref{prop:pointwise_conv_equilibrium}, the limit is an equilibrium. It is not $0$ but it
might be different from $g^*$. We will use Assumption \ref{Assum_connectivity} to ensure that $0$
and $g^*$ are the only equilibria. In order to prove this result, we need the following lemma.

\begin{lemma}[Instantaneous propagation of the infection]\label{lem:QuasiInt}
 Suppose Assumptions \ref{Assum1} and \ref{Assum_connectivity} are in force. If $g \in \Delta$ is such that:
  \[
  \int_\Omega g(x) \, \mu(\mathrm{d}x) > 0.
  \]
  Then, for all $t>0$, $\phi(t,g)$ is $\mu$-a.e. positive.
\end{lemma}
\begin{proof}
  Since the flow is order-preserving (see Proposition \ref{prop:Ordpres}), it is sufficient to show
  the proposition for $g$ such that $\norm{g} < 1/2$. It follows from Equation
  \eqref{eq:semiflow_int} that:
  \[
    \phi(t,g) \leq \norm{g} + t \norm{T_k} .
  \]  
  Thus, for all $t \in [0, c)$, with $c=(1-2\norm{g})/2\norm{T_k}$ (and $c=+\infty $ if
  $\norm{T}=0$), we have that $\phi(t,g) <1/2$. Now, we define the function:
  \[
    u(t) = \mathrm{e}^{-\norm{\gamma} t}\, \mathrm{e}^{t T_k/2} g.
  \]  
  We get the following inequality for all $t\in[0,c)$:
  \[
    u'(t) - (T_k/2 - \norm{\gamma})(u(t)) =0\leq (1/2- \phi(t,g)) T_k (\phi(t,g)) \leq \partial_t
    \phi(t,g) - (T_k/2 - \norm{\gamma})(\phi(t,g)).
  \]
  Using Corollary \ref{cor:Comparison} with $v(t)=\phi(t,g)$ and $F=T_k/2 - \gamma$ (which is
  clearly cooperative as it is linear), we get for $t\in [0, c)$:
  \begin{equation}
    \label{eq:minoration}
    \phi(t,g) \geq u(t). 
  \end{equation}
  Now, we fix $t \in [0,c)$. We denote by $A=\set{ x \in \Omega \,
    \colon \, u(t)(x) > 0}$ the support of $u(t)$. 
  We have:
  \[
    0 = \braket{\mathds{1}_{A^c} , u(t)} = \mathrm{e}^{-\norm{\gamma} t}\, \sum\limits_{n \in
      \mathbb{N}} \frac{1}{n!}\braket{\mathds{1}_{A^c}, (t T_k/2)^n(g)}.
  \]
  This implies that $\braket{\mathds{1}_{A^c}, (t T_k/2)^n(g)} = 0$ for all $n$, and thus that
  $\braket{\mathds{1}_{A^c} , T_k u(t) } = 0$. We deduce that:
  \[
    \int_{A^c \times A} k(x,y) \, \mu(\mathrm{d}x)\mu(\mathrm{d}y) = 0.
  \]
  Since the set $A$ contains the support of $g$, we get $\mu(A)>0$. It follows from Assumption
  \ref{Assum_connectivity} that $\mu(A^c) = 0$. This means that $u(t) $ is $\mu$-a.e.  positive.
  Hence, from Equation \eqref{eq:minoration}, we get that, for $t \in [0,c)$, $\phi(t,g)$ is
  $\mu$-a.e. positive.  Using the semi-group property of the semi-flow this results propagates on
  the whole positive half-line and the result is proved.
\end{proof}
\begin{remark}\label{rem:positivity_equilibrium}
  Lemma \ref{lem:QuasiInt} together with Remark \ref{rem:equilibrium_not_almost_zero} shows that an
  equilibrium $h^*$ different from $0$ is $\mu$-a.e. positive.
\end{remark}
\begin{remark}
  One can check from its proof, that Lemma \ref{lem:QuasiInt} does not require the integrability
  condition \eqref{eq:strongintegrability} in Assumption \ref{Assum1} to be true.
\end{remark}

Now we can show the following important result.

\begin{proposition}[Uniqueness of the endemic state]
  \label{prop:uniq}
  Under Assumptions \ref{Assum1} and \ref{Assum_connectivity}, the maximal equilibrium $g^*$ is the
  unique equilibrium different from $0$.
\end{proposition}
\begin{proof}
  Let $h^*$ be another equilibrium different from $0$. Since $g^*$ is the maximal equilibrium, we
  have $h^* \leq g^*$. We shall prove that $h^*$ is equal to $g^*$ almost everywhere. Let us define
  the non-negative kernel $\mathsf{k}$ by:
  \[
    \mathsf{k}(x,y)=(1- g^*(x)) \, \frac{k(x,y)}{\gamma(y)} \quad \text{for $x,y\in \Omega$.}
  \]
  Notice that $\mathsf{k}$ satisfies \eqref{eq:strongintegrability2}.  Since
  $T_\mathsf{k} (\gamma g^*) = \gamma g^*$, we deduce from Proposition \ref{prop:Collatz} that
  $r(T_\mathsf{k}) \geq 1$. Let $v\in L^q(\mu)_+ \backslash \set{0}$ be a left Perron vector of the
  operator $T_\mathsf{k}$ (given by Lemma \ref{lem:schae}~\ref{lem:schaeIV}). The kernel
  $\mathsf{k}$ satisfies Assumption \ref{Assum_connectivity} as $k$ does and $1-g^*$ is positive
  everywhere (see Remark \ref{rem:g_star_inf_1}). Hence, $v$ can be chosen positive
  $\mu$-a.e. according to Lemma \ref{lem:schae}~\ref{lem:schaeVI}. The following computation:
  \[
    \braket{v, \gamma g^*} = \braket{v, T_\mathsf{k}(\gamma g^*)} = r(T_\mathsf{k}) \braket{v, g^*},
  \]
  shows that $r(T_\mathsf{k})$ is actually equal to $1$ since $\braket{v, \gamma g^*}>0$. Now we
  compute:
  \begin{align*}
    0 &= \braket{v, F(h^*)}  \\
      &= \braket{v, T_\mathsf{k}(\gamma h^*) - \gamma h^*} + \braket{v, (g^* -
        h^*) T_{k / \gamma} (\gamma h^*)} \\ 
      &= \braket{v, (g^* - h^*) T_{k}(h^*)},
  \end{align*}
  where we used that $\braket{v, T_\mathsf{k}f - f}=0$ as $r(T_\mathsf{k})=1$ and $v$ is a left
  Perron eigenvector.  According to Remark \ref{rem:positivity_equilibrium}, $h^*$ is
  $\mu$-a.e. positive. Since we have $T_{k}(h^*)=\gamma h^*/(1-h^*)$, the function $T_{k}(h^*)$ is
  also $\mu$-a.e. positive. Hence $g^*$ and $h^*$ are equal $\mu$-a.e. since $v$ is
  $\mu$-a.e. positive, see Lemma \ref{lem:schae}~\ref{lem:schaeVI}. This implies in particular that
  $T_k(h^*)=T_k(g^*)$ by Lemma \ref{lem:schae}~\ref{lem:schae0}. We deduce that, for all
  $x \in \Omega$:
  \[
    h^*(x) = T_k(h^*)(x)/ (\gamma(x) + T_k(h^*)(x)) = T_k(g^*)(x)/ (\gamma(x) + T_k(g^*)(x)) =
    g^*(x).
  \]
  Therefore $g^*$ is then unique equilibrium different from $0$.
\end{proof}

Now we can prove the main result of this section on the pointwise convergence of $\phi(t,g)$. If $g$
is $\mu$-a.e. equal to 0, then clearly, as $\gamma$ is positive, we get that
$\lim\limits_{t \to \infty} \phi(t, g) =0$ pointwise. So only the case $g$ not $\mu$-a.e. equal to
$0$ is pertinent.
\begin{theorem}\label{th:super}
  Suppose that Assumptions \ref{Assum1} and \ref{Assum_connectivity} are in force. Let
  $g \in \Delta$ such that $\int_\Omega g(x) \, \mu(\mathrm{d}x) > 0$.  Then,we have that for all
  $x \in \Omega$:
  \[
    \lim\limits_{t \to \infty} \phi(t, g)(x) = g^*(x).
  \]
\end{theorem}
\begin{proof}
  By Lemma \ref{lem:QuasiInt}, it is enough to show the result for $g$ $\mu$-a.e. positive. We
  define, for all $n \in \mathbb{N}^*$:
  \[
    \Omega_n = \set{x \in \Omega \, \colon \, g(x) \geq 1/n}.
  \]
  Since $g$ is $\mu$-a.e. positive, we get that
  $\lim_{n\rightarrow\infty } \left(1- \frac{1}{n}\right) \mathds{1}_{\Omega_n}=1$ in $L^1(\mu)$.
  Besides, $R_0$ is greater than $1$ by Proposition \ref{prop:R0_s}. Hence, according to Proposition
  \ref{prop:spectrumcontinuous}, there exists $n$ large enough such that:
  \[
    R_0\left(\left( 1 - 1/n\right) \mathds{1}_{\Omega_n}\right) > 1.
  \]
  By applying Proposition \ref{prop:R0_s}~\ref{prop:R0_sIII} to the kernel
  $( 1 - 1/n) \mathds{1}_{\Omega_n}(x) k(x,y)$, we get that there exists
  $w_n \in \mathscr{L}^\infty_+ \backslash \set{0}$ and $\lambda>0$ such that:
  \[
    \left( 1 - \frac{1}{n}\right) \mathds{1}_{\Omega_n} T_{k}(w_n) = (\gamma + \lambda) w_n.
  \]
  We deduce that for all $x \in \Omega_n^c$, $w_n(x) =0$. Furthermore we can choose $w_n$ such that
  $\norm{w_n}\leq 1/n$. This proves that $w_n \leq g$. Then, it follows from the monotony of the
  semi-flow (see Proposition \ref{prop:Ordpres}) that, for all $t \in \mathbb{R}_+$:
  \begin{equation}
    \label{eq:encadrement}
    \phi(t, w_n) \leq  \phi(t, g) \leq \phi(t, 1).
  \end{equation}
  Besides, we have:
  \begin{align*}
    0 \leq \lambda w_n  
    &= ( 1 - 1/n) \mathds{1}_{\Omega_n} T_{k}(w_n) - \gamma w_n \\
    &\leq ( 1 - 1/n) T_{k}(w_n) - \gamma w_n \\
    &\leq( 1 - w_n) T_{k}(w_n) - \gamma w_n \\
    &=  F(w_n),
  \end{align*}
  where the last inequality follows from the fact that $\norm{w_n} \leq 1/n$. Thus, the path
  $t \mapsto \phi(t, w_n)$ is non-decreasing according to Proposition \ref{prop:propC}. Hence, it
  converges pointwise to a limit $h^* \neq 0$ since
  $w_n \in \mathscr{L}^\infty_+ \backslash \set{0}$. This limit has to be an equilibrium by
  Proposition \ref{prop:pointwise_conv_equilibrium}. Since $0$ and $g^*$ are the only equilibria by
  Proposition \ref{prop:uniq}, we have necessarily $h^* = g^*$. We conclude thanks to Equation
  \eqref{eq:encadrement}.
\end{proof}

\subsection{Endemic states in the critical regime}
\label{sec:exple-N}
Here we show by a counter-example that the integral condition \eqref{eq:strongintegrability} is
necessary to obtain the convergence towards the disease-free equilibrium in the critical regime. In
the following example, the transmission kernel has a bounded density with respect to a finite
measure $\mu$ and we have $\inf\gamma>0$ and $R_0=1$. However, there exists a continuum set of
distinct equilibria.  \medskip

Consider the set $\mathbb{N}^*$ equipped with some finite measure $\mu$ such that
$\mu_n=\mu(\set{n})>0$ for all $n \in \mathbb{N}^*$. We choose $\gamma$ constant equal to $1$ and
the kernel $\kappa$ defined for $i,j\in \mathbb{N}^*$ by:
\begin{equation}
  \label{eq:counter_example}
  \kappa(i,\set{j}) = \left\{
  \begin{array}{ll}
    \frac{2 i + 2}{2 i - 1} & \mbox{if } j = i+1, \\
    0 & \mbox{otherwise},
  \end{array} \right. \qquad \text{and} \qquad \gamma(i) = 1.
\end{equation}
Clearly Assumption \ref{Assum0} is satisfied. Moreover, the kernel $\kappa$ has the following
density $k$ with respect to $\mu$ defined by $k(i,j)=\kappa(i,\set{j})/\mu(\set{j})$ for
$i,j\in \mathbb{N}^*$.  However condition \eqref{eq:weakintegrability}, and thus
\eqref{eq:strongintegrability} from Assumption \ref{Assum1}, is not satisfied.  Indeed, for all
$q>1$ we have:
\[
\sup\limits_{n \in \mathbb{N}^*}  \int_{\mathbb{N}^*} k(x,y)^{q} \,  
  \mu(\mathrm{dy})  
=\sup\limits_{n \in \mathbb{N}^*} k(n,n+1)^q \, \mu_{n+1} =
\lim\limits_{n\to \infty} \frac{(2 n + 1)^q}{(2n-1)^q}
\, \mu_{n+1}^{1-q} = +\infty, 
\]
where divergence of the sequence follows from the convergence of $\mu_{n+1}$ to $0$ (because $\mu$
is a finite measure). The following proposition asserts that we are in the critical regime.
\begin{proposition}
  Let $\kappa$ be defined by \eqref{eq:counter_example}, $k$ be its density and $\gamma=1$.  We have
  for the reproduction number: $R_0=r(T_{k/\gamma})=1$, and for the spectral bound:
  $s(T_k - \gamma)=0$.
\end{proposition}
\begin{proof}
  Since $\gamma$ is the function constant equal to $1$, we have $s(T_k - \gamma) = R_0 - 1$ and
  $R_0 = r(T_k)$. We compute the spectral radius of $T_k$ using Gelfand's formula:
  \[
    r(T_k) = \lim\limits_{n \to \infty} \norm{T_k^n}^{1/n} = \lim\limits_{n \to \infty}
    \left(\prod\limits_{i = 1}^n \frac{2 i + 2}{2 i - 1} \right)^{1/n} = 1.
  \]
  The limit is found by applying the logarithm to the sequence and using Ces\`aro lemma.
\end{proof}

The following result shows that even if we are in the critical regime, the maximal equilibrium $g^*$
is not equal to $0$ everywhere, and there exists infinitely many distinct equilibria. For
$\alpha\in [0,1]$, we define the function $g_\alpha^*$ on $\mathbb{N}^*$ by $g^*_\alpha(1) = \alpha$
and for $n \in \mathbb{N}^*$:
\[
  g^*_\alpha(n+1) =
  \begin{cases}
    \frac{2n - 1}{2n+2}\,\, \frac{ g^*_\alpha(n)}{1 - g^*_\alpha(n)}
    & \text{if $g^*_\alpha(n)<1$,}\\
    0 & \text{if $g^*_\alpha(n)\geq 1$.}
  \end{cases}
\]

\begin{proposition}\label{prop:equilibria_example}
  Let $\kappa$ be defined by \eqref{eq:counter_example}, $k$ be its density and $\gamma=1$.
\begin{propenum}
\item\label{ex-eqI} The equilibria of Equation \eqref{eq:SIS} are
  $\set{g^*_\alpha\, \colon \, \alpha\in [0, 1/2]}$.
\item\label{ex-eqII} The function $\alpha\mapsto g^*_\alpha$ defined on $[0, 1/2]$ and taking values
  in $\Delta\subset \mathscr{L}^\infty$ is increasing and continuous (with respect to
  $\norm{\cdot}$).  In particular, the set of equilibria is totally ordered, compact and connected.
\item The equilibrium $g_{1/2}^*$ is the maximal equilibrium. We have that $g_{1/2}^*(n)=1/(2n)$ for
  $n \in \mathbb{N}^*$.
\end{propenum}
\end{proposition}

\begin{proof}
  We first explicit $g^*_{1/2}$ and prove property \ref{ex-eqII}. Let $\Gamma$ denote the function
  $\alpha\mapsto g^*_\alpha$ defined on $[0, 1/2]$ and taking values in $ \mathscr{L}^\infty$.  By
  definition of $g^*_\alpha$, it is immediate that $g_{1/2}^*(n)=1/(2n)$ and $g^*_0(n)=0$ for
  $n \in \mathbb{N}^*$.  Using that the function $x \mapsto \lambda x / (1 - x)$ is increasing on
  $[0, 1)$ for all $\lambda> 0$, we deduce by induction that
  $0\leq g_\alpha^*(n)< g_\beta^*(n) \leq g^*_{1/2}(n)$ for all $0\leq\alpha < \beta\leq 1/2$ and
  $n\in \mathbb{N}^*$.  This implies that the function $\Gamma$. As $g^*0$ and $g^*_{1/2}$ belong to
  $\Delta$, we deduce that $\Gamma$ takes values in $\Delta$ by monotonicity.  It is also immediate
  to check that the function $\Gamma$ is continuous for the pointwise convergence in $\Delta$. Then
  using that
  $\lim_{n\rightarrow \infty } \sup_{\alpha\in [0, 1/2]} g_\alpha^*(n)=\lim_{n\rightarrow \infty }
  g_{1/2}^*(n)=0$, we deduce the function $\Gamma$ is also continuous with respect to the uniform
  convergence in $\Delta$. This proves property \ref{ex-eqII}.

  \medskip

  We prove property \ref{ex-eqI}.  It is clear that if $h^*$ is an equilibrium, then $h^*(n)<1$ for
  all $n \in \mathbb{N}^*$ thanks to Remark \ref{rem:g_star_inf_1} and by the definition of the
  kernel $\kappa$ that:
  \begin{equation}
    \label{eq:exple-h*}
    h^*(n+1) = \frac{2n - 1}{2n+2}\,\, \frac{ h^*(n)}{1 - h^*(n)}
    \quad\text{for all $n\in     \mathbb{N}^*$.}
  \end{equation}
  This readily implies that $g^*_\alpha$ is an equilibrium for $\alpha\in [0, 1/2]$ as, in this
  case, $g^*_\alpha(n)\leq g^*_{1/2}(n)=1/(2n)$ and $g^*_\alpha$ solves \eqref{eq:exple-h*}. As
  $g^*_1(1)=1$, we also get that $g^*_1$ is not an equilibrium.

  Let $\alpha\in (1/2, 1)$. We shall now prove by contradiction that there exists
  $n\in \mathbb{N}^* $ such that $g^*_\alpha(n)\geq 1$.  Let us assume that $g^*_\alpha(n)< 1$ for
  all $n\in \mathbb{N}^* $. Arguing as in the first part of the proof, we get
  $g^*_\alpha(n) > g^*_{1/2} (n)$ for all $n\in \mathbb{N}^*$. Thus the sequence
  $v=(v_n\,\colon\, n\in \mathbb{N}^*)$ with $v_n = 2n g^*_\alpha(n)$ satisfies the following
  recurrence for $ n\in \mathbb{N}^*$:
  \[
    v_{n+1} = v_n \frac{2n - 1}{2n - v_n} \quad\text{and} \quad 1<v_n<2n.
  \]
  We deduce that the sequence $v$ is increasing, and thus
  $v_{n+1} \geq v_n \frac{2n - 1}{2n - 2\alpha} $, as $v_1=2\alpha$.  We deduce that
  $v_n\geq c \, n^{\alpha -1/2}$ for some positive constant $c$.  This in turn implies that
  $v_{n+1} \geq v_n \frac{2n - 1}{2n - c \, n^{\alpha -1/2}}$ and thus
  $v_n \geq c' \exp(c'' n^{\alpha -1/2})$ for some positive constants $c'$ and $c''$.  This
  contradicts the fact that $v_n<2n$ for $ n\in \mathbb{N}^*$.  As a conclusion, there exists
  $n\in \mathbb{N}^* $ such that $g^*_\alpha(n)\geq 1$.  This implies that $g ^*_\alpha$ can not be
  an equilibrium. This ends the proof of property \ref{ex-eqI}.  \medskip

  We have already computed $g_{1/2}^*$. We deduce from properties \ref{ex-eqI} and \ref{ex-eqII}
  that $g_{1/2}^*$ is the maximal equilibrium.
\end{proof}

Since $k$ is upper-triangular, the long-time behavior of the dynamic does not depend on the first
terms of the initial condition. Indeed, for $n \geq 2$, consider the subspace
$E_n =\set{g\in \mathscr{L}^\infty\,\colon\, g(p)=0 \text{ for $1\leq p<n$}}$ of functions whose
first $n-1$ terms are $0$. Denote by $P_n$ the canonical projection from $\mathscr{L}^\infty$ on
$E_n$.  For $n \geq 2$ and $g \in \Delta$, we have:
\begin{equation}
   \label{eq:def-Pn}
   P_n \phi(t,g) = P_n\left(\phi(t, P_n(g))\right).
\end{equation}
Let us introduce a new partial order $\preceq$ defined by $g \preceq h$ if there exists $n \geq 2$
such that $P_n(g) \leq P_n(h)$. We have the following result.
\begin{proposition}
  For all $g \in \Delta$, we have:
  \[
    g^*_{\alpha_-} \leq \liminf\limits_{t \to \infty} \phi(t,g) \leq \limsup\limits_{t \to \infty}
    \phi(t,g) \leq g^*_{ \alpha_+}
  \]
  with
  \[
    \alpha_-= \max \set{\alpha \, \colon \, g^*_\alpha \preceq g} \quad\text{and}\quad \alpha_+ =
    \min \set{\alpha \, \colon \, (g \wedge g^*_{1/2} )\preceq g^*_\alpha}.
  \]
\end{proposition}
\begin{proof}
  Using \eqref{eq:def-Pn}, it is easy to check that
  $\lim_{t\rightarrow +\infty } \phi(t, P_n g^*_\alpha)=g^*_\alpha$ for $n\geq 2$.  Using that the
  flow is order preserving, we get that if $g^*_\alpha \preceq g \preceq g^*_\beta$ for some
  $0\leq \alpha\leq \beta\leq 1/2$, then:
  \[
    g^*_\alpha \leq \liminf\limits_{t \to \infty} \phi(t,g) \leq \limsup\limits_{t \to \infty}
    \phi(t,g) \leq g^*_\beta.
  \]
  The result then follows from the continuity and the monotonicity of $\alpha \mapsto g^*_\alpha$
  (for $\alpha\in [0, 1/2]$) and that $g^*_{1/2}$ is the maximal equilibrium (see Proposition
  \ref{prop:equilibria_example}).
\end{proof}

\subsection{Uniform convergence}
In Sections \ref{subsec:critical} and \ref{subsec:supercritical}, we have obtained results about
pointwise convergence toward the equilibrium $g^*$. The next result implies in particular that this
convergence is uniform if $\inf \gamma>0$. (See the stronger result from Theorem \ref{th:sub} in the
sub-critical case, where the uniform convergence is exponentially fast.)
\begin{theorem}\label{th:strongconv}
  Suppose that Assumption \ref{Assum1} and \ref{Assum_connectivity} are in force and let
  $A \in \mathscr{F}$. If $\gamma$ is bounded away from $0$ on $A$, that is:
  \[
    \inf\limits_{x \in A} \gamma(x) > 0,
  \]
  then, for $g \in \Delta$, with positive integral if $g^*\neq 0$, we have:
  \[
    \lim\limits_{t \to \infty}
    \,\, \sup\limits_{x \in A} \, \abs{\phi(t, g)(x) -g^*(x)} = 0. 
  \]
\end{theorem}
\begin{proof}
  Set $m=\inf\limits_{x \in A} \gamma(x)$.  For $s\in \mathbb{R}_+$, we have:
  \begin{align*}
    \partial_t (\phi(s, 1)-g^*) 
    &= F(\phi(s, 1)) - F(g^*) \\
    &\leq  (1 - \phi(s, 1)) T_k(\phi(s, 1) - g^*) - \gamma (\phi(s, 1) - g^*) \\
    &\leq T_k(\phi(s, 1) - g^*) - \gamma (\phi(s, 1) - g^*)\\
    &\leq M\, \norm{\phi(s, 1) - g^*}_p - \gamma (\phi(s, 1) - g^*),
  \end{align*}
  where we used that $T_k$ is positive for the second inequality and H\"{o}lder inequality for the
  last with
  $M= \sup\limits_{x \in \Omega} \left(\int_{\Omega} k(x,y)^{q} \, \mu(\mathrm{dy}) \right)^{1/q} <
  \infty$.  For $s\in \mathbb{R}_+$, set $v_s= \mathrm{e}^{ms}(\phi(s, 1)-g^*) $. Notice that
  $v_s\geq 0$ and that $ \partial_t v_s(x) \leq M\, \norm{v_s}_p$ for $x\in A$. Integrating for
  $s\in [0, t]$, we deduce that for $x\in A$:
  \begin{align*}
    0\leq  (\phi(t, 1)-g^*)(x)
    &\leq  \mathrm{e}^{-mt}(1- g^*) + M\int_0^t \mathrm{e}^{-
      m(t-s)}\norm{\phi(s, 1) - g^*}_p
      \, \mathrm{d}s\\
    &\leq  \mathrm{e}^{-mt} + M\int_0^t \mathrm{e}^{-
      ms}\norm{\phi(t-s, 1) - g^*}_p
      \, \mathrm{d}s.
  \end{align*}
  Note the right hand-side does not depend on $x$. As $\phi(s,1)$ converges pointwise to $g^*$ (see
  Equation \eqref{eq:gstar}) and is bounded by $1$, using the dominated convergence theorem, we
  deduce that the right hand-side goes to $0$ as $t$ goes to infinity. So, we obtain that:
  \begin{equation}
    \label{eq:cvu1}
    \lim_{t\rightarrow +\infty  } \,  \sup\limits_{x \in A}  \, \abs{\phi(t,
      1)(x) -g^*(x)} = 0.
  \end{equation}
  If $g^*=0$, use that $0\leq \phi(t,g)\leq \phi(t, 1)$ for all $g\in\Delta$ and $t\in \mathbb{R}_+$
  to conclude.

  \medskip

  If $g^*$ is non zero (which corresponds to the super-critical case), consider a function $f\leq g$
  with positive integral such that $f\leq g^*$. By monotonicity of the flow, this implies that
  $0\leq g^* - \phi(s,f)$ for all $s\in \mathbb{R}_+$. Arguing similarly as above, we get for
  $s\in \mathbb{R}_+$:
  \[
    \partial_t (g^* -\phi(s, f)) \leq M\, \norm{g^* -\phi(s, f)}_p - \gamma (g^* -\phi(s, f)).
  \]
  Using that $\phi(s,f)$ converges pointwise to $g^*$ (see Theorem \ref{th:super}), we similarly get
  that
  \begin{equation}
    \label{eq:cvuf}
    \lim_{t\rightarrow +\infty  } \,  \sup\limits_{x \in A}  \, \abs{\phi(t,
      f)(x) -g^*(x)} = 0.
  \end{equation}
  Then, use the monotonicity of the flow which implies that
  $\phi(t,f)\leq \phi(t, g) \leq \phi(t, 1)$ for $f\leq g\leq 1$ as well as \eqref{eq:cvu1} and
  \eqref{eq:cvuf} to conclude.
\end{proof}


\section{Vaccination model}\label{sec:vacc}

\subsection{Infinite-dimensional models}

We write an infinite-dimensional model that take into account the heterogeneity in the transmission
of the infectious disease in the spirit of \eqref{eq:SIS} which generalizes Equations
\eqref{eq:1D_leaky} and \eqref{eq:1D_all_or_nothing} and take into account a family of different
vaccines.  Recall that the measurable space $(\Omega, \mathscr{F})$ represents the features of the
individuals in a given population, the finite measure $\mu$ describes the size of the population and
its sub-groups, and the number $\gamma(x)$ is the recovery rate of individuals with feature
$x \in \Omega$.  The transmission kernel $\kappa$ describes the way the disease is spread among the
population without vaccination.

Suppose that we have different vaccines or treatments available that we can give to individuals in
order to fight the disease upstream. The set of vaccines is represented by a set $\Sigma$ which is
finite in practice. We endow $\Sigma$ with a $\sigma$-field $\mathscr{G}$. We are also given two
measurable functions $e, \delta \, \colon \,\Omega\times \Sigma \to [0,1]$. For both models, the
number $\delta(x,\xi)$ is the relative reduction of infectiousness for people with feature $x$
vaccinated by the vaccine $\xi$. The coefficient $e(x,\xi)$ is the efficacy of vaccine $\xi$ given
on individuals with feature $x$. In $\Sigma$ there is a particular type of vaccines $\xi_0$ which is
the absence of vaccination.  This vaccination has no efficacy upon the individuals: $e(x,\xi_0) = 0$
and $\delta(x,\xi_0) = 0$ for all $x \in \Omega$.  We define a vaccination policy as a non-negative
kernel $\eta \, \colon \, \Omega \times \mathscr{G} \to [0,1]$. The probability for an individual
with feature type $x$ to be vaccinated by a vaccine in the measurable set $A \in \mathscr{G}$ under
the policy $\eta$ is equal to $\eta(x, A)$.  The recovery rate can be affected by the vaccine, and
in this case $\gamma$ is then a non-negative measurable function defined on $ \Omega\times \Sigma $,
with $\gamma(x, \xi)$ the recovery rate of individuals with feature $x$ and vaccine $\xi$.  The
number $u(t, x, \xi)$ is the probability for an individual with feature $x$ which has been
inoculated by the vaccine $\xi$ to be infected at time $t$.  The total number of infected
individuals at time $t$ is therefore given by:
\begin{equation}
  \int_{\Omega \times \Sigma} u(t,x, \xi) \, \eta(x, \mathrm{d}\xi) \mu(\mathrm{d}x).
\end{equation}

\subsubsection{The leaky vaccination mechanism} In this setting, $e(x, \xi)$ denotes the leaky
vaccine efficacy of $\xi \in \Sigma$ on an individual with feature $x$, \textit{i.e.}, the relative
reduction in the transmission rate.  We generalize Equation \eqref{eq:1D_leaky} to get the following
infinite dimensional evolution equation:
\begin{multline}\label{eq:leaky_multi}
  \partial_t u(t, x, \xi) = - \gamma(x, \xi) \, u(t, x, \xi) \\
  + (1 - u(t, x, \xi))(1 - e(x, \xi)) \int_{\Omega \times \Sigma} (1 - \delta(y,\zeta)) u(t,y,\zeta)
  \kappa(x,\mathrm{d}y) \eta(y,\mathrm{d}\zeta).
\end{multline}
The evolution Equation \eqref{eq:leaky_multi} can be seen as the SIS evolution Equation
\eqref{eq:SIS} with:
\begin{itemize}
\item[-] feature $\boldsymbol{x}=(x,\xi)$ and feature space
  $\boldsymbol{\Omega }=\Omega \times \Sigma$ endowed with the $\sigma$-field
  $\mathscr{F} \otimes \mathscr{G}$,
\item[-] recovery rate: $\boldsymbol{\gamma}(\boldsymbol{x})=\gamma(x, \xi)$,
\item[-] transmission kernel:
\begin{equation}
   \label{eq:leaky-kernel}
   \boldsymbol{\kappa}^a(\boldsymbol{x}, \mathrm{d}\boldsymbol{y})=
   (1 - e(x, \xi) ) (1 - \delta(y, \zeta)) \kappa(x,\mathrm{d}y)
   \eta(y,\mathrm{d}\zeta).
 \end{equation}
\end{itemize}

\begin{remark}
  In the leaky mechanism, we suppose that the vaccine acts directly on the susceptibility and the
  infectiousness of the individuals. Protective gears (like respirators or safety glasses) which are
  designed to protect the wearer from absorbing airborne microbes or transmitting them have a
  similar effect. Hence, Equation \eqref{eq:leaky_multi} is not limited to vaccination and can also
  be used as a model for distribution of equipment in the population.
\end{remark}

\subsubsection{The all-or-nothing mechanism} In this setting, $e(x, \xi)$, is defined as the
probability to immunize completely the individual with feature $x$ to the disease with vaccine
$\xi$. We generalize Equation \eqref{eq:1D_all_or_nothing} to get the following infinite dimensional
evolution equation:
\begin{multline}\label{eq:aon_multi}
  \partial_t u(t, x, \xi) = - \gamma(x, \xi)\, u(t, x, \xi) \\ + (1 - e(x, \xi) - u(t, x, \xi))
  \int_{\Omega \times \Sigma} (1 - \delta(y, \zeta)) u(t,y,\zeta) \kappa(x,\mathrm{d}y)
  \eta(y,\mathrm{d}\zeta).
\end{multline}
The probability $v(t, x, \xi)=u(t,x, \xi)/ (1 - e(x,\xi))$ for an individual with feature $x$ which
has not been vaccinated by the inoculation of vaccine $\xi$ to be infected at time $t$ satisfies the
following equation:
 \begin{multline}\label{eq:aon_multi2}
   \partial_t v(t, x, \xi) = - \gamma(x) v(t, x, \xi) \\ + (1 - v(t, x, \xi)) \int_{\Omega \times
     \Sigma} (1 - \delta(y,\zeta)) v(t,y,\zeta) (1 - e(y, \zeta) ) \kappa(x,\mathrm{d}y)
   \eta(y,\mathrm{d}\zeta).
\end{multline}
The evolution Equation \eqref{eq:aon_multi2} can be seen as the SIS evolution Equation
\eqref{eq:SIS} with:
\begin{itemize}
\item[-] feature $\boldsymbol{x}=(x,\xi)$ and feature space
  $\boldsymbol{\Omega }=\Omega \times \Sigma$ endowed with the $\sigma$-field
  $\mathscr{F} \otimes \mathscr{G}$,
\item[-] recovery rate: $\boldsymbol{\gamma}(\boldsymbol{x})=\gamma(x, \xi)$,
\item[-] transmission kernel:
\begin{equation}
   \label{eq:aon-kernel}
   \boldsymbol{\kappa}^\ell(\boldsymbol{x}, \mathrm{d}\boldsymbol{y})=
   (1 - e(y, \zeta) ) (1 - \delta(y, \zeta)) \kappa(x,\mathrm{d}y)
   \eta(y,\mathrm{d}\zeta).
 \end{equation}
\end{itemize}
Notice the difference between the evolution Equation \eqref{eq:leaky_multi} for leaky mechanism and
the evolution Equation \eqref{eq:aon_multi2} for the all-or-nothing mechanism is that $e(y, \zeta)$
in \eqref{eq:aon_multi2} (or in the kernel $\boldsymbol{\kappa}^a$ from \eqref{eq:aon-kernel}) is
replaced by $e(x, \xi)$ in \eqref{eq:leaky_multi} (or in the kernel $\boldsymbol{\kappa}^\ell$
from\eqref{eq:leaky-kernel}).

\subsection{Discussion on the basic reproduction number}
Suppose that Assumption \ref{Assum1} is in force. Then, we can define a new basic reproduction
number for the vaccination models. We consider the following bounded operators on
$\mathscr{L}^\infty(\Omega \times \Sigma)$:
\begin{align*}
T(g)(x,\xi) 
&= \int_{\Omega \times \Sigma} (1 - \delta(y, \zeta)) g(y,
\zeta) \, \frac{\kappa(x,\mathrm{d}y)}{\gamma(y, \zeta)}  \,  \eta(y,
  \mathrm{d}\zeta), \\ 
\\ M(g)(x, \xi) &= (1 - e(x, \xi)) g(x, \xi).
\end{align*}
Following Section \ref{subsec:R0}, the all-or-nothing vaccination reproduction number $R_0^a(\eta)$
associated to Equation \eqref{eq:aon_multi2} and vaccine policy $\eta$ is:
\begin{equation}
   \label{eq:all-R0}
   R_0^a(\eta) = r(TM).
\end{equation}
where we recall that $r$ stands for the spectral radius. For the leaky vaccination, the basic
reproduction number associated to Equation \eqref{eq:leaky_multi} and vaccine policy $\eta$ is:
\begin{equation}
   \label{eq:leaky-R0}
R_0^\ell(\eta) = r(MT).
\end{equation}
In \cite{shim_distinguishing_2012}, the authors already remarked that the two vaccination mechanisms
actually leads to the same basic reproduction number for the one-group models.  This result also
holds in the infinite-dimension SIS model. Notice that Assumption \ref{Assum1} insures that those
two basic reproduction numbers are well defined.
\begin{proposition}
  \label{prop:aon=leaky}
  We assume Assumption \ref{Assum1} holds. Let $\eta$ be a vaccination policy.  The basic
  reproduction number for the leaky vaccination and the for the all-or-nothing vaccination are the
  same:
  \[
    R_0^\ell(\eta) = R_0^a(\eta).
  \]
\end{proposition}
\begin{proof}
  Thanks to the definition of the spectral radius \eqref{eq:spectral_radius} and the basic
  reproduction numbers defined in \eqref{eq:all-R0} and \eqref{eq:leaky-R0}, the result is a direct
  consequence of the following equality on the spectra:
  \[
    \sigma(MT) \cup \set{0} = \sigma(TM) \cup  \set{0}.
  \]
  We prove this later equality by following \cite[Appendix~A1]{pedersen_c-algebras_2018}. Let
  $\lambda \in \mathbb{C} \backslash (\sigma(MT) \cup \set{0})$. By definition, there exists a
  bounded operator $A$ on $\mathscr{L}^\infty(\Omega \times \Sigma)$ such that:
  \[
    A (\lambda \mathrm{Id} - MT) = (\lambda \mathrm{Id} - MT) A = \mathrm{Id},
  \]
  where $\mathrm{Id}$ is the identity operator. Then, one can check easily that
  $\lambda^{-1} (\mathrm{Id} + T A M)$ is the inverse of $\lambda \mathrm{Id} - TM$, whence
  $\lambda \in \mathbb{C} \backslash (\sigma(TM) \cup \set{0})$.  This gives that
  $\sigma(TM) \cup \set{0}) \subset \sigma(MT) \cup \set{0})$. The other inclusion is proved
  similarly.
\end{proof}

\subsection{The perfect vaccine}
The most simplistic case is a situation where there is only one vaccine with complete efficacy on
every individuals: $\Sigma = \set{\xi_0, \xi_1}$ with $e(x,\xi_1) = 1$ and $\delta(x,\xi_1) = 1$ for
all $x \in \Omega$. Recall that $\xi_0$ corresponds to the absence of vaccine. For simplicity, we
denote by $\eta^0(x) = \eta(x, \set{\xi_0})$ the probability for or the proportion of individuals of
type $x \in \Omega$ which are not vaccinated. We assume for simplicity that initially no vaccinated
individuals are infected, that is $u(0,x, \xi_1)=0$. Since individuals that have been vaccinated are
fully immunized, we have $u(t,x, \xi_1) = 0$ for all $x$ and $t$. The only equation that matter is
the one on $u^0(t,x)=u(t,x, \xi_0)$ which represents the proportion of unvaccinated individuals that
are infected. For both mechanisms (all-or-nothing and leaky vaccination), the evolution equation of
$u^0$ writes:
\begin{equation}
  \label{eq:SIS-vaccinated}
  \partial_t u^0(t, x) = (1  - u^0(t, x)) \int_{\Omega}
  u^0(t,y)  \eta^0(y) \kappa(x,\mathrm{d}y) - \gamma(x) u^0(t, x).
\end{equation}
We shall use this formulation in a future work to find optimal vaccination policies for a given
cost.


\section{Limiting contacts within the population}\label{sec:quarantine}
Motivated by the recent lockdown policies taken by many countries all around the world to slow down
the propagation of Covid-19 in 2020, we propose to investigate the possible impact on our SIS model
of the limitations of contacts within the population. We consider the case where $\kappa$ takes the
form of Example \ref{ex:graphonform}:
\[
  \kappa_W(x,\mathrm{d} y) = \beta(x) W(x,y) \theta(y) \, \mu(\mathrm{d} y),
\]
where $\beta$ is the susceptibility function, $\theta$ is the infectiousness function, $\mu$ is a
probability measure on the space $\Omega$ of features of the individuals and the graphon $W$
represents the initial graph of the contacts between individuals of the population (recall that
$W(x,y)=W(y,x)\in [0, 1]$ is the probability that $x$ and $y$ are connected and can be also seen as
the density of contact between the individuals with features $x$ and $y$).  In order to stress the
dependence in $W$, we write $R_0(W)=r(T_{\kappa_W/\gamma})$ the corresponding basic reproduction
number and $\phi_W$ the semi-flow \eqref{eq:semiflow} associated to $F=F_W$ in \eqref{eq:DefF} given
by $F_W(g)=(1-g) T_{\kappa_W}(g) - \gamma g$.  We model the impact of a policy which reduces the
contacts between the individuals, by a new graph of contact given by a new graphon $W'$. We say that
$W'$ is a perfect lockdown with respect to $W$ if:
\begin{equation}
  \label{eq:perfect_lockdown}
  W'(x,y) \leq W(x,y), \qquad \forall x,y \in \Omega.
\end{equation}

Intuitively $x$ and $y$ have a lesser probability to be connected in the graphon $W'$ than in the
graphon $W$.  We get the following intuitive result as a direct application of Theorem
\ref{th:spectralradius}~\ref{th:spectralradiusII} and Corollary \ref{cor:Comparison}.
\begin{proposition}[Perfect Lockdown]
  Assume that $\beta$ and $\theta$ are bounded and $\gamma$ is bounded away from 0.  If $W'$ is a
  perfect lockdown with respect to $W$ then $R_0(W') \leq R_0(W)$ and
  $\phi_{W'}(t,g) \leq \phi_{W}(t,g)$ for all initial condition $g \in \Delta$.
\end{proposition}

However, assuming that all the contacts within the population are reduced might be unrealistic
(\textit{e.g.}  people can have stronger contacts with their family in lockdown).  Instead, we can
suppose as a weaker condition, that each individual reduces the average number of contacts he has.
Recall \eqref{eq:deg} for the definition of the degree $\mathrm{deg}_W(x)$ of an individual
$x \in \Omega$ (\textit{i.e.} the average number of his contacts) and the mean degree
$\mathrm{d}_{W}$ over the population for a graphon $W$.  as:
\[
  \mathrm{deg}_W(x) = \int_\Omega \!\! W(x,y) \, \mu(\mathrm{d}y) \quad\text{and}\quad \mathrm{d}_W
  = \int_\Omega \!\! \mathrm{deg}_W(x) \, \mu(\mathrm{d}x)= \int_{\Omega^2} \!\! W(x,y) \,
  \mu(\mathrm{d}y)\, \mu(\mathrm{d}x).
\]
Recall that $\norm{\cdot}_1$ is the usual $L^1(\mu)$ norm. The following lemma bounds the basic
reproduction number with the supremum and the mean degree of the graphon.

\begin{lemma}\label{lem:bounded-R0}
  Let $W$ be a graphon. Assume that $\beta$ and $\theta/\gamma$ are bounded. We have that:
\begin{equation}\label{eq:bounded-R0}
  \frac{1}{\norm{\gamma/\beta\theta}_1} \, \mathrm{d}_{W} \leq R_0(W)
  \leq \norm{\beta\theta/\gamma} \, \sup\limits_{x \in \Omega} \, \mathrm{deg}_{W}(x).
\end{equation}
\end{lemma}
\begin{proof}
  Recall $T_\mathsf{k}$ is the operator defined by \eqref{eq:def-T} with
  $\kappa(x, \mathrm{d}y)= \mathsf{k}(x,y) \, \mu(\mathrm{d}y)$. Let $M(v)$ be the operator
  corresponding to the multiplication by the function $v$.  We have:
  \begin{align*}
    R_0(W) 
    &= r(M(\beta) \, T_{W}\,  M(\theta/\gamma)) \\
    &= r(M(\beta\theta/\gamma)\,  T_{W} )\\
    &\leq \norm{M(\beta\theta/\gamma) \, T_{W} } \\
    &= \sup\limits_{x \in \Omega} \frac{\beta(x)\theta(x)}{\gamma(x)}
      \, \int_\Omega W(x,y) \, \mu(\mathrm{d}y) \\
    &\leq \norm{\beta\theta/\gamma} \, \sup\limits_{x \in \Omega}\,
      \mathrm{deg}_{W}(x),
  \end{align*}
  where we used the definition of the basic reproduction number \eqref{eq:R0} for the first
  equality, arguments similar as in the proof of Proposition \ref{prop:aon=leaky} for the second,
  and the (third) definition of the spectral radius \eqref{eq:spectral_radius} for the first
  inequality.

  \medskip
  
  Using similar arguments, we have:
  \[
    R_0(W) = r(M(\beta) \,T_W\, M(\theta/\gamma)) = r\left( M(v) \, T_W\, M(v)\right),
  \]
  with $v=\sqrt{\beta \theta/\gamma}$.  Recall notations from Lemma \ref{lem:schae}, and notice that
  $ M(v) \, T_W\, M(v)=T_\mathsf{k}$ is a bounded integral operator on $\mathscr{L}^\infty $
  associated to the symmetric kernel $\mathsf{k}(x,y) = v(x)W(x,y) v(y)$.  According to Lemma
  \ref{lem:schae}~\ref{lem:schaeIII-r} with $q = p = 1/2$ and $\hat T_\mathsf{k}$ the integral
  operator on $L^2(\mu)$ with the same kernel $\mathsf{k}$, defined in \eqref{eq:def-hat-Tk}, we get
  $R_0(W)=r(\hat T_\mathsf{k})$.  The operator $\hat T_\mathsf{k}$ is self-adjoint, as $\mathsf{k}$
  is symmetric, and compact according to \ref{lem:schaeIII}. Thanks to the Courant-Fischer-Weyl
  min-max principle, we obtain:
  \[
    R_0(W) =r(\hat T_\mathsf{k})= \sup_{g\in L^2(\mu)\backslash\set{0}} \frac{\braket{ M(v)\, g,\,
        T_W\, M(v)\, g}}{\braket{g,g}}\cdot
  \]
  Taking $g = 1/v$, we get $M(v)g=1$ and thus:
  \[
    R_0(W) \geq \frac{\braket{1, T_W1}}{\norm{\gamma/\beta \theta}_1} =
    \frac{d_W}{\norm{\gamma/\beta \theta}_1}\cdot
  \]
  This ends the proof of Lemma \ref{lem:bounded-R0}.
\end{proof}

We deduce from Lemma \ref{lem:bounded-R0} the following result for a lockdown policy $W'$ for which
the degree of each individuals is less than the average degree of the initial graphon $W$.

\begin{corollary}[Partial Lockdown]
  \label{cor:lockdown}
  Assume that $\beta$ and $\theta/\gamma$ are bounded. If $W'$ is a partial lockdown of $W$, that
  is:
  \begin{equation}
    \label{eq:lockdown}
    \sup\limits_{x \in \Omega} \, \mathrm{deg}_{W'}(x) \leq C \mathrm{d}_{W} \quad\text{with}\quad
    C=\frac{1}{\norm{\beta\theta/\gamma}\,\norm{\gamma/\beta\theta}_1}, 
  \end{equation}
  then we have $R_0(W') \leq R_0(W)$.
\end{corollary}
In the general case, we have $C \leq 1$.  But, if the functions $\beta$, $\theta$ and $\gamma$ are
constants (or simply if $\beta\theta/\gamma$ constant), then we have $C=1$ since $\mu$ is a
probability measure.
\begin{remark}
  \label{rem:W=p}
  Suppose that $\beta$, $\theta$ and $\gamma$ are constants (or that $\beta\theta/\gamma$ is
  constant).  Inequality \eqref{eq:bounded-R0} shows that the graphon $W$ which corresponds to a
  minimal basic reproduction number $R_0(W)$, when the mean degree $\mathrm{d}_{W}$ is fixed, say
  equal to $p$, is any graphon with constant degree equal to $p$, that is $\mathrm{deg}_{W}(x) =p$
  for all $x\in \Omega$.  We then deduce from Lemma \ref{lem:bounded-R0} that
  $R_0(W)= p\beta\theta/\gamma$.

  \medskip

  This is in particular the case for the constant graphon $W=p\in [0,1] $.  According to Example
  \ref{ex:graphingform}\ref{ex:graphonform-c}, this corresponds to the one dimensional SIS model
  \eqref{eq:one-group}.  \medskip

  This is also the case for the geometric graphon, where the probability of edges between $x$ and
  $y$ depends only on the distance between $x$ and $y$.  Keeping notations from Example
  \ref{ex:graphingform}\ref{ex:graphonform-geom}, we consider the population uniformly spread on the
  unit circle: $\Omega=[0, 2\pi]$ and $\mu(\mathrm{d}x)=\mathrm{d}x/2\pi$, and the graphon $W_f$
  defined by $W_f(x,y)=f(x-y)$ for $x, y\in \Omega$, where $f$ is a measurable non-negative function
  defined on $\mathbb{R}$ which is bounded by 1 and $2\pi$-periodic.  Let
  $p=(2\pi)^{-1}\, \int_{[0, 2\pi]} f(y)\, \mathrm{d}y$.  We have:
  $\mathrm{deg}_{W}(x) = \mathrm{d}_{W}=p$; the basic reproduction number
  $R_0(W_f)=p\beta\theta/\gamma$ and the maximal equilibrium $g^*=\max(0, 1 - R_0^{-1})$.
  Furthermore, the graphon $W_f$ minimizes the basic reproduction number among all graphons with
  mean degree $p$. It is interesting to notice that $R_0(W_f)$ does not depend on the support of $f$
  or even on $\sup\{|r| \, \colon\, r\in [-\pi, \pi] \text{ and } f(r)>0\}$, which can be seen as
  the maximal contamination distance from an infected individual.
\end{remark}


\appendix

\section{Proofs of the results of Theorem \ref{th:Inv} and Corollary \ref{cor:Comparison}}
\label{sec:proofs}

Let $X$ be a Banach space. For $x \in X$ and $D \subset X$, we denote by $\rho(x, D)$ the distance
between $x$ and the set $D$:
\begin{equation}
  \label{eq:DistanceSet}
  \rho(x, D) = \inf\limits_{y \in D} \norm{x - y}.
\end{equation}

Let $a>0$ and $G \, \colon \, [0,a) \times X \to X$ be a locally Lipschitz function with respect to
the second variable. Recall Definition \ref{def:forwardinvariant} of a forward invariant set with
respect to $G$. The following result appears in \cite[Theorem~5.2]{deimlingordinary}.
\begin{lemma}\label{Invariance2}
  Let $D$ be a closed convex set with non-empty interior. Suppose that $G$ satisfies:
  \begin{equation}
    \label{eq:lim}
    \lim\limits_{\lambda \to 0^+} \frac{1}{\lambda} \rho(x + \lambda G(t,x), D) = 0,
    \qquad \forall (t,x) \in (0,a)\times\partial D.
  \end{equation} 
  Then $D$ is forward invariant with respect to $G$.
\end{lemma}
If the set $D$ is a proper cone with non-empty interior, the following equivalence enables to
establish \eqref{eq:lim} more easily. It is a consequence of \cite[Lemma~4.1]{deimlingordinary} and
\cite[Example~4.1.ii]{deimlingordinary}
\begin{lemma}
  \label{Invariance3}
  Let $K$ be a proper cone and let $x \in \partial K$ and $z\in X$. The following conditions are
  equivalent:
\begin{propenum}
\item $\lim\limits_{\lambda \to 0^+} \lambda^{-1} \, \rho(x + \lambda z, K) = 0.$
\item For all $x^\star\in K^\star$ such that $\braket{x^\star,x} = 0$, we have
  $\braket{x^\star,z} \geq 0$.
\end{propenum}
\end{lemma}
Thanks to these two lemmas, we give the demonstration of the result of forward invariance of Section
\ref{subsec:preamble}:
\begin{proof}[Proof of Theorem \ref{th:Inv}]
  Let $y \in X$. We assume that, for all $(x,t) \in \partial K \times [0,a)$ and for all
  $x^\star \in K^\star$ such that $\braket{x^\star,x} = 0$, we have:
  $\braket{x^\star,G(t,y+x)} \geq 0$. According to Lemma \ref{Invariance3}, we obtain:
  \[ \lim\limits_{\lambda \to 0^+} \lambda^{-1} \, \rho(x + \lambda G(t,y + x), K) = 0, \] for all
  $(x,t) \in \partial K \times [0,a)$. Since
  $\rho(y + x + \lambda G(t,y + x), y + K) = \rho(x + \lambda G(t,y + x), K)$ by Equation
  \eqref{eq:DistanceSet}, we can conclude the proof using Lemma \ref{Invariance2} with $D=y+K$.
\end{proof}

We end this section with the proof of the comparison theorem.

\begin{proof}[Proof of Corollary \ref{cor:Comparison}]
  We suppose that $F$ is cooperative on $D_1 \times X$ and the inequality \eqref{eq:Comparison}
  holds. Let $w = v - u$.  The function $w$ is solution of the ODE $w' = G(t,w)$ where:
  \[
    G(t,x) = F(u(t)+x) - F(u(t)) + d(t) \qquad \text{and} \qquad d(t) = v'(t) - F(v(t)) - u'(t) + F(u(t)).
  \]
  First we show that $G$ is locally Lipschitz with respect to the second variable. Let
  $(t,x) \in [0,a) \times X$. Let $U$ be a neighborhood of $u(t)+x$ such that $F$ is Lipschitz on
  $U$ with a Lipschitz constant $L$. By continuity of $u$, there exist a neighborhood $V_x$ of $x$
  and a positive constant $\eta$, such that $u(s) + y \in U$, for all $s \in [t,t+\eta] \cap [0,a)$
  and $y \in V_x$. Thus, for all $s \in [t,t+\eta] \cap [0,a)$ and $y,z \in V_x$, we have
  $\norm{G(s, y) - G(s, z)} \leq L \norm{y - z}$.
  
  Let $t \in [0,a), x \in \partial K$ and let $x^\star \in K^\star$ such that
  $\braket{x^\star,x} = 0$. Let us prove that $\braket{x^\star,G(t,x)} \geq 0$. By
  \eqref{eq:Comparison}, we know that $d(t)$ belongs to $K$. Furthermore, the inequality
  $\braket{x^\star,F(u(t) + x) - F(u(t))}\geq 0$ holds because the function $F$ is cooperative on
  $D_1 \times X$.  Thus, $\braket{x^\star,G(t,x)}$ is non-negative. Hence, we can apply Theorem
  \ref{th:Inv} with $y=0$ and obtain that $K$ is forward invariant with respect to $G$. Since
  $w(0) \in K$, this shows that $w(t) \in K$ for all $t \in [0,a)$, \textit{i.e.}, $u(t) \leq v(t)$
  for all $t \in [0,a)$.
  
  When $F$ is cooperative on $X \times D_2$, the proof is similar.
\end{proof}

\section{The Hausdorff distance on the compact sets of $\mathbb{C}$}\label{sec:Hausdorff}

Let $\mathcal{K}(\mathbb{C})$ be the set of non-empty compact subsets of $\mathbb{C}$. The Hausdorff
distance between $A$ and $B$ in $\mathcal{K}(\mathbb{C})$ is defined as:
\begin{equation}
  d_{\mathrm{H}}(A, B) = \max \left\lbrace \sup\limits_{z_1 \in A} \; \inf\limits_{z_2 \in B} \, \abs{z_1 - z_2}, \quad  \sup\limits_{z_2 \in B} \inf\limits_{z_1 \in A}  \abs{z_2 - z_1} \right\rbrace.
\end{equation}
We recall that the space $(\mathcal{K}(\mathbb{C}), d_{\mathrm{H}})$ is a metric space, see
\cite[Section 7.3.1]{burago2001}.  Since
$\sup \set{ \abs{z} \, \colon \, z \in A} = d_{\mathrm{H}}(A, \set{0})$ for all
$A \in \mathcal{K}(\mathbb{C})$, we deduce the following result.
\begin{lemma}\label{lem:Hausdorff_radius}
  The map $A \mapsto \sup \set{ \abs{z} \, \colon \, z \in A}$ from
  $(\mathcal{K}(\mathbb{C}), d_{\mathrm{H}})$ to $\mathbb{R}$ endowed with the usual Euclidean
  distance is continuous.
\end{lemma}


\bibliographystyle{plain}
\bibliography{biblio}

\end{document}